\documentclass[reqno, a4paper, 10pt]{amsart}
\usepackage{amssymb,url, color, mathrsfs}
\usepackage[colorlinks=true, bookmarks=true, pdfstartview=FitH, pagebackref=true, linktocpage=true, linkcolor = magenta, citecolor = blue]{hyperref}
\usepackage[short,nodayofweek]{datetime}

\theoremstyle{plain}
\numberwithin{equation}{section}
\newtheorem{theorem}{Theorem}
\newtheorem{lemma}{Lemma}

\newtheorem{proposition}{Proposition}
\theoremstyle{remark}

\DeclareMathOperator{\Rset}{\mathbf{R}}
\DeclareMathOperator{\vol}{\textrm{vol}}

\newcommand{\Sscr}{\mathscr S}
\newcommand{\Cscr}{\mathscr C}
\newcommand{\Nscr}{\mathscr N}

\def\cfac#1{\ifmmode\setbox7\hbox{$\accent"5E#1$}\else\setbox7\hbox{\accent"5E#1}\penalty 10000\relax\fi\raise 1\ht7\hbox{\lower1.05ex\hbox to 1\wd7{\hss\accent"13\hss}}\penalty 10000\hskip-1\wd7\penalty 10000\box7 }

\author[Q.A. Ng\^o]{Qu\cfac oc Anh Ng\^o}
\address[Q.A. Ng\^o]{Department of Mathematics, College of Science\\
Vi\^{e}t Nam National University\\
H\`{a} N\^{o}i, Vi\^{e}t Nam}
\email{nqanh@vnu.edu.vn,  bookworm\_vn@yahoo.com}

\author[V.H. Nguyen]{Van Hoang Nguyen}
\address[V.H. Nguyen]{School of Mathematical Sciences\\
Tel Aviv University\\
Tel Aviv 69978, Israel.}
\email{vanhoang0610@yahoo.com}

\thanks{V.H. Nguyen's current address: Institut de Math\'ematiques de Toulouse, Universit\'e Paul Sabatier, 31062 Toulouse c\'edex 09, France.}

\dedicatory{Dedicated to Professor Ho\`{a}ng Qu\^{o}\hspace{0.5ex}\llap{\raise 0.5ex\hbox{\'{}}}\hspace{-0.5ex}c To\`{a}n on the occasion of his 70th birthday}

\begin{document} 

\allowdisplaybreaks

\title[Sharp reversed HLS inequality on $\Rset^n$]
{Sharp reversed Hardy--Littlewood--Sobolev inequality on $\Rset^n$}

\begin{abstract}
This is the first in our series of papers that concerns Hardy--Littlewood--Sobolev (HLS) type inequalities. In this paper, the main objective is to establish the following sharp reversed HLS inequality in the whole space $\Rset^n$
\[\int_{\Rset^n} \int_{\Rset^n} f(x) |x-y|^\lambda g(y) dx dy \geqslant \Cscr_{n,p,r} \|f\|_{L^p(\Rset^n)}\, \|g\|_{L^r(\Rset^n)}\]
for any non-negative functions $f\in L^p(\Rset^n)$, $g\in L^r(\Rset^n)$, and $p,r\in (0,1)$, $\lambda > 0$ such that $1/p + 1/r -\lambda /n =2$. We will also explore some estimates for $\Cscr_{n,p,r}$ and the existence of optimal functions for the above inequality, which will shed light on some existing results in literature.
\end{abstract}

\date{\bf \today \; at \, \currenttime}



\maketitle
\tableofcontents

\section{Introduction}

Of importance in quantitative theories of differential equations are the so-called Sobolev inequalities. Generally, these inequalities provide an estimate of lower order derivatives of a function in terms of higher order derivatives. Such an estimate is an essential tool in other areas of mathematical analysis including calculus of variation, geometric analysis, etc. Let us recall the following sharp fractional Sobolev inequality
\begin{equation}\label{eqFractionalSobolevInequality}
\left( \int_{\Rset^n} |u|^\frac{2n}{n-2s} dx \right)^{1-2s/n} \leqslant \Sscr_{n,s} \|u\|_{W^s(\Rset^n)}^2
\end{equation}
for all $u \in W^s(\Rset^n)$ where $s \in (0,n/2)$. The best constant $\Sscr_{n,s}$ in \eqref{eqFractionalSobolevInequality} is computed as
\begin{equation}\label{eqBestConstantFractionalSobolevInequality}
\Sscr_{n,s}= \frac{{\Gamma (n/2-s)}}{{{2^{2s}}{\pi ^s}\Gamma (n/2+s)}}{\left( {\frac{{\Gamma (n)}}{{\Gamma (n/2)}}} \right)^{2s/n}}
\end{equation}
and the equality in \eqref{eqFractionalSobolevInequality} occurs if and only if $u(x)=c(1+|(x-x_0)/t|^2)^{s-n/2}$ for some $t>0$, $c \in \Rset$ and $x_0 \in \Rset^n$. Concerning the best constant $\Sscr_{n,s}$, it was first computed by Rosen \cite{Rosen71} in the case $s=1$ and $n=3$. For general $n \geqslant 3$ and with $s=1$, the best constant $\Sscr_{n,1}$ was computed independently by Aubin \cite{aubin76} and Talenti \cite{talenti76}. For general $s \in (0, n/2)$, the best constant $\Sscr_{n,s}$ was given by Lieb in \cite{l1983} when he considered the sharp constant of Hardy--Littlewood--Sobolev (HLS) inequalities which will be mentioned later.

In existing literature, the classical HLS inequality named after Hardy and Littlewood \cite{hl1928,hl1930} and Sobolev \cite{sobolev1938} on $\Rset^n$ states that for any $n\geqslant 1$, $p,r > 1$ and $\lambda \in (0,n)$ satisfying $1/p + 1/r + \lambda /n =2$, there exists a constant $\Nscr_{n,\lambda,p} > 0$ such that
\begin{equation}\label{eq:HLSineq}
\bigg |\int_{\Rset^n}\int_{\Rset^n} \frac{f(x)g(y)}{|x-y|^\lambda} dx dy\bigg| \leqslant \Nscr_{n,\lambda,p} \|f\|_{L^p(\Rset^n)} \|g\|_{L^r(\Rset^n)},
\end{equation}
for any $f\in L^p(\Rset^n)$ and $g\in L^q(\Rset^n)$.

From \cite[Theorem $4.3$]{liebloss2001}, it is well-known that the sharp constant $\Nscr_{n,\lambda,p}$ satisfies the following estimate
\[ 
\Nscr_{n,\lambda,p} \leqslant {\frac{n}{n-\lambda}} {\left(\frac{\pi^{\lambda/2}}{\Gamma(1 + n/2)}\right)^{\lambda/n}} {\frac{1}{pr}} {\left({\left(\frac{\lambda p}{n (p-1)}\right)^{\lambda/n}} + {\left(\frac{\lambda r}{n (r-1)}\right)^{\lambda/n}}\right)}
\]
while in the diagonal case $p=r=2n/(2n-\lambda)$ (or one of these parameters is $2$), it follows from the seminal work \cite{l1983} that 
\[
\Nscr_{n,\lambda,p} =\Nscr_{n,\lambda} = {\pi ^{\lambda /2}}\frac{{\Gamma (n/2 - \lambda /2)}}{{\Gamma (n -  \lambda /2)}}{\left( {\frac{{\Gamma (n)}}{{\Gamma (n/2)}}} \right)^{1-\lambda /n}} .
\]
The existence of optimal functions to \eqref{eq:HLSineq} was also proven by Lieb in \cite{l1983} by using symmetric rearrangement arguments. Generally speaking, the equality in \eqref{eq:HLSineq} occurs if and only if $f(x)=g(x)=c(1+|(x-x_0)/t|^2)^{\lambda/2-n}$ for some $t>0$, $c \in \Rset$ and $x_0 \in \Rset^n$, up to a constant multiple. Recently, it has been found that the sharp HLS inequality \eqref{eq:HLSineq} can be proven without using symmetric rearrangement arguments; for interested readers, we refer to \cite{ccl2010,fl2010,fl2012}.

It is quite a surprise to note that the Sobolev and HLS inequalities are dual for certain families of parameters. To see this more precise, we let $\lambda = n-2s$ in \eqref{eq:HLSineq} and rewrite the right hand side of \eqref{eq:HLSineq} with ${2^{ - 2s}}{\pi ^{ - n/2}} \Gamma (n/2-s)/\Gamma (s)$, which is the Green function of the operator $(-\Delta)^s$ in $\Rset^n$ for each $s \in (0,n/2)$ to get
\begin{equation}\label{eqHLS->Sobolev}
\int_{\Rset^n} f (-\Delta)^{-s}(f) dx \leqslant \Sscr_{n,s} \|f\|_{L^{2n/(n+2s)}(\Rset^n)}^2
\end{equation}
Hence, the sharp HLS inequality implies the sharp Sobolev inequality. Further seminal works reveal that the sharp HLS inequality can also imply the Moser--Trudinger--Onofri inequality, the logarithmic HLS inequality \cite{beckner1993}, as well as the Gross logarithmic Sobolev inequality \cite{gross1975}. All these inequalities have many important applications in analysis, geometry, and quantum field theory. 

In the last two decades, HLS inequality \eqref{eq:HLSineq} has captured the attention of many mathematicians. Some remarkable extensions have already been drawn. For example, one has HLS inequalities on the upper half space $\Rset_+^n$, on Heisenberg groups, on compact Riemannian manifolds, and on weighted forms; for interested readers, we refer to \cite{dz2013,fl2012a,hanzhu,st1958}. 

Apart from these extensions, Dou and Zhu \cite{dz2014} recently discovered the following \textit{reversed HLS inequality on $\Rset^n$} which can be seen as an extension of \eqref{eq:HLSineq} for negative $\lambda$.

\begin{theorem}[reversed HLS inequality on $\Rset^n$]\label{thmMAIN}
Let $p,r\in (0,1)$ and $\lambda > 0$ such that $1/p + 1/r -\lambda /n =2$. Then there exists a positive constant $C(n,p,r)$ such that for any non-negative functions $f\in L^p(\Rset^n)$ and $g\in L^r(\Rset^n)$, we have
\begin{equation}\label{eq:RHLS}
\int_{\Rset^n} \int_{\Rset^n} f(x) |x-y|^\lambda g(y) dx dy \geqslant C(n,p,r) \|f\|_{L^p(\Rset^n)} \|g\|_{L^r(\Rset^n)}.
\end{equation}
\end{theorem}

Note that in \cite[Theorem 1.1]{dz2014}, the authors require $p, r \in (n/(n+\lambda),1)$ instead of $p,r \in (0,1)$ as shown above. However, by resolving the condition $1/p + 1/r -\lambda /n =2$, it is not hard to see that indeed $p,r$ must satisfy $p, r \in (n/(n+\lambda),1)$. Hence, it is safe to assume $p,r \in (0,1)$. Concerning inequality \eqref{eq:RHLS}, it is worth noting that it has been applied to solve some curvature equations with negative critical Sobolev exponents by Zhu in \cite{zhu2014}. As can be easily seen, the proof given in \cite{dz2014} is purely based on an extension of the classical Marcinkiewicz interpolation theorem applying to the singular integral operator defined by
$$(I_\lambda f)(x) = \int_{\Rset^n} f(y) |x-y|^\lambda dy.$$
It was proven in \cite{dz2014} that $I_\lambda f$ fulfills the following estimate
$$\|I_\lambda f\|_{L^q(\Rset^n)} \geqslant \Cscr \|f\|_{L^p(\Rset^n)},$$
for some constant $\Cscr>0$ where $q =r/(r-1) \in (-\infty,0)$.

The primary aim of this paper is to provide an alternative proof for the reversed HLS inequality \eqref{eq:RHLS} which follows the standard idea in the proof of the classical HLS inequality \eqref{eq:HLSineq} given in \cite{liebloss2001}. This alternative proof is more concise than that of Dou and Zhu and does not use the Marcinkiewicz-type interpolation technique. As we shall see later, our proof also gives us an explicit bound from below for the constant $C(n,p,r)$ in \eqref{eq:RHLS}; see \eqref{eqConstantC_npr} for details.

Once we establish Theorem \ref{thmMAIN}, it is natural to ask whether or not the optimal functions for the reversed HLS inequality \eqref{eq:RHLS} exist. For this purpose, we will turn our attention to consider the following minimizing problem
\begin{equation}\label{eq:variationalprob}
\Cscr_{n,p,r} :=\inf_f \big\{ \|I_\lambda f\|_{L^q(\Rset^n)} : f\geqslant 0, \|f\|_{L^p(\Rset^n)} =1 \big\}.
\end{equation}
Obviously, $\Cscr_{n,p,r} \geqslant 0$ and is finite. In addition, we can easily verify that optimal functions for the reversed HLS inequality \eqref{eq:RHLS} are those solving the problem \eqref{eq:variationalprob}. 

The existence of optimal functions for \eqref{eq:variationalprob} was proven by Dou and Zhu \cite{dz2014} for the diagonal case $p=r=2n/(2n+\lambda)$. To establish such a result, the authors follows the idea in \cite{l1983}, which is based on rearrangement arguments.

In this paper, we will also address the existence of optimal functions for \eqref{eq:variationalprob}, however, in full generality of parameters by relaxing the restriction $p=r=2n/(2n+\lambda)$; that is, we consider \eqref{eq:variationalprob} for all $p,r\in (0,1)$ satisfying $1/p + 1/r -\lambda /n =2$. We will also show that, up to a translation, all optimal functions of \eqref{eq:variationalprob} are radially symmetric and strictly decreasing. We shall prove the following result.

\begin{theorem}\label{Existence}
There exists some non-negative function $f\in L^p(\Rset^n)$ such that $\|f\|_{L^p(\Rset^n)} =1$ and $\|I_\lambda f\|_{L^q(\Rset^n)} = \Cscr_{n,p,r}$. Moreover, if $f$ is a minimizer of \eqref{eq:variationalprob} then there exist a non-negative, strictly decreasing function $h$ on $[0,\infty)$ and some $x_0\in \Rset^n$ such that $f(x) = h(|x+x_0|)$ a.e. $x\in \Rset^n$.
\end{theorem}

Let us now consider the diagonal case $p=r=2n/(2n+\lambda)$ for which the sharp constant $\Cscr_{n,p,r}$ can be explicitly computed. Inspired by \cite[Theorem 1.2']{dz2014}, we will prove the following sharp reversed HLS inequality.

\begin{theorem}\label{thmOPTIMAL}
Let $\lambda > 0$, then for any non-negative functions $f\in L^{2n/(2n+\lambda)} (\Rset^n)$ and $g\in L^{2n/(2n+\lambda)} (\Rset^n)$ we have
\begin{equation}\label{eq:RHLS-diagonal}
\int_{\Rset^n} \int_{\Rset^n} f(x) |x-y|^\lambda g(y) dx dy \geqslant \Cscr_{n,\lambda} \|f\|_{L^{2n/(2n+\lambda)} (\Rset^n)} \|g\|_{L^{2n/(2n+\lambda)} (\Rset^n)}.
\end{equation}
where
\[
\Cscr_{n,\lambda} ={\pi ^{\lambda /2}}\frac{{\Gamma (n/2 - \lambda /2)}}{{\Gamma (n - \lambda /2)}}{\left( {\frac{{\Gamma (n)}}{{\Gamma (n/2)}}} \right)^{1-\lambda /n}} .
\]
with constant $\Cscr_{n,\lambda}$ sharp.
\end{theorem}

As a consequence of Theorem \ref{eq:RHLS-diagonal} and inspired by \cite{CarlenLoss92}, we will formally derive a  \textit{reversed log-HLS inequality}. It is clear that the existence of an optimal function pair for \eqref{eq:RHLS-diagonal} follows from Theorem \ref{Existence}. Moreover, if $(f,g)$ is an optimal function pair of \eqref{eq:RHLS-diagonal} then, up to a translation, $f$ and $g$ are radially symmetric and strictly decreasing by means of Lemma \ref{lemmaexistence}. By simple calculation, up to a multiplicative constant, the pair $(f,g)$ must satisfy the following system
\begin{equation}\label{eq:Euleur}
\left\{
\begin{split}
|f(x)|^{-\lambda/(2n+\lambda)-1} f(x) = \int_{\Rset^n} |x-y|^{\lambda} g(y) dy,\\
|g(y)|^{-\lambda/(2n+\lambda)-1} g(x) = \int_{\Rset^n} |x-y|^{\lambda} f(x) dx;
\end{split}
\right.
\end{equation}
see Section \ref{secCLASSIFICATION}. Let $u =|f|^{-\lambda/(2n+\lambda)-1} f$ and $v = |g|^{-\lambda/(2n+\lambda)-1}g$ in \eqref{eq:Euleur}, it leads us to study positive solutions of the following system of integral equations
\begin{equation}\label{eq:systemDouZhu}
\left\{
\begin{split}
u(x) &= \int_{\Rset^n} |x-y|^\lambda v(y)^\kappa dy,\\
v(y) &= \int_{\Rset^n} |x-y|^\lambda u(x)^\kappa dx,
\end{split}
\right.
\end{equation}
in $\Rset^n$ where we denote $\kappa = -(2n+\lambda)/\lambda < 0$. Note that the integral system \eqref{eq:systemDouZhu} is well-known to be conformal invariant; hence, one can adopt the method of moving spheres  to classify measurable solutions of \eqref{eq:systemDouZhu}. 

In the literature, the method of moving spheres, introduced by Li and Zhu in \cite{lz1995}, is a variant of the well-known method of moving planes, introduced by Aleksandrov in \cite{a1958}. For interested readers, we refer to \cite{s1971, gnn1979, cgs1989, cl1991, clo2005, clo2006} for the method of moving planes and its variants, while for the method of moving spheres we refer to \cite{l2004, xu2005}. 

In the last part of Dou and Zhu's work \cite{dz2014}, the authors showed that any non-negative, measurable solution $(u,v)$ of \eqref{eq:systemDouZhu} must be of the following form
$$u(x) = v(x) = (1+|x|^2)^{\lambda/2},$$
up to translations and dilations.

Motivated by the above classification, in the last part of this paper, we will also classify solutions of integral systems of the form \eqref{eq:systemDouZhu} where $\kappa$ is no longer $-(2n+\lambda)/\lambda$. To be precise, we are interested in the classification of non-negative, measurable functions of the following system
\begin{equation}\label{eqIntegralSystem}
\left\{
\begin{split}
u(x) &= \int_{\Rset^n} |x-y|^p v(y)^{-q} dy,\\
v(x) &= \int_{\Rset^n} |x-y|^pu(y)^{-q} dy,
\end{split}
\right.
\end{equation}
in $\Rset^n$ with $p, q>0$. We shall prove the following result.

\begin{theorem}\label{thmCLASSIFICATION}
For $n \geqslant 1$, $p>0$ and $q>0$, let $(u,v)$ be a pair of non-negative Lebesgue measurable functions in $\Rset^n$ satisfying \eqref{eqIntegralSystem}. Then $q=1+2n/p$ and, for some constants $a, b>0$ and some $\overline x \in \Rset^n$, $u$ and $v$ take the following form
\[
u(x)=v(x) = a(b^2 + |x-\overline x|^2)^{p/2}
\]
for any $x \in \Rset^n$.
\end{theorem}

Once we prove Theorem \ref{thmCLASSIFICATION}, we can go back to prove Theorem \ref{thmOPTIMAL} and obtain the sharp constant $\Cscr_{n, \lambda}$. Before closing this section, it is worth noting that in our next article \cite{NgoNguyen2015}, we will perform the same study for the case of the half space $\Rset_+^n$.

\section{The reversed HLS inequality on $\Rset^n$: Proof of Theorem \ref{thmMAIN}}

In this section, we provide an alternative proof of the reversed HLS inequality \eqref{eq:RHLS}. As mentioned before, our proof here is completely different from the one in \cite{dz2014} which mimics the same idea from the proof of the classical HLS inequality given in \cite{liebloss2001}.

In order to prove \eqref{eq:RHLS}, we first set up some notation and conventions. For each point $x\in \Rset^n$, let us denote
\[
B_c(x) = \{y\in \Rset^n\, :\, |y-x| \leqslant c\}.
\]
In the special case $x = 0$, we simply denote $B_c(0)$ by $B_c$; hence $B_c = \{y\in \Rset^n\, :\, |y| \leqslant c\}$. For $a,b,c > 0$, we denote 
$$u(a) = |\{f > a\}|,\quad v(b) = |\{g > b\}|,$$
where $|A|$ denotes the Lebesgue measure of the measurable subset $A\subset \Rset^n$. By homogeneity, we can normalize $f$ and $g$ in such a way that $\|f\|_{L^p(\Rset^n)}= \|g\|_{L^r(\Rset^n)}=1$. Therefore, we have
$$p\int_0^\infty a^{p-1} u(a) da = \|f\|_{L^p(\Rset^n)}^p = 1$$
and 
$$r \int_0^\infty b^{r-1} v(b) db = \|g\|_{L^r(\Rset^n)}^r = 1.$$
For simplicity, we denote
$$I(f,g) = \int_{\Rset^n} \int_{\Rset^n} f(x) |x-y|^\lambda g(y) dx dy.$$
The layer cake representation \cite[Theorem 1.13]{liebloss2001} implies that
$$f(x) = \int_0^\infty \chi_{\{f> a\}}(x) da,\quad g(y) = \int_0^\infty \chi_{\{g> b\}}(y) dy,$$
and
$$|x-y|^{\lambda} = \lambda \int_0^\infty c^{\lambda-1} \chi_{\{\Rset^n\setminus B_c\}}(x-y) dc.$$
For simplicity, we also denote
\[
J(a,b,c) = \int_{\Rset^n} \int_{\Rset^n} \chi_{\{f > a\}}(x) \chi_{\{\Rset^n\setminus B_c\}}(x-y) \chi_{\{g> b\}}(y) dx dy.
\]
Then the Fubini theorem tells us that
\begin{equation}\label{eq:step1}
I(f,g) = \lambda \int_0^\infty\int_0^\infty\int_0^\infty c^{\lambda-1} J(a,b,c) da db dc.
\end{equation}

\noindent\textbf{Step 1}. Our first step to prove \eqref{eq:RHLS} is to claim the following: There holds
\begin{equation}\label{eq:claim}
J(a,b,c) \geqslant  u(a)v(b)/2,
\end{equation}
for any $c$ satisfying 
\[
2\omega_n c^n \leqslant \max\{u(a), v(b)\},
\]
where $\omega_n$ denotes the volume of $B_1$. To verify \eqref{eq:claim}, we let $u(a) \geqslant v(b)$, then
\begin{align*}
J(a,b,c) &= \int_{\Rset^n} \chi_{\{g> b\}}(y) \, |\{f > a\} \cap (\Rset^n\setminus B_c(y) )| dy \\
&= \int_{\Rset^n} \chi_{\{g> b\}}(y) \, (|\{f > a\}| -|\{f > a\} \cap B_c(y)|) dy \\
&\geqslant \int_{\Rset^n} \chi_{\{g> b\}}(y) \, (u(a) -|B_c(y)|) dy \\
&=\int_{\Rset^n} \chi_{\{g> b\}}(y) \, (u(a) -c^n \omega_n) dy \\
&\geqslant u(a) v(b)/2.
\end{align*}
Repeating the same argument shows that $J(a,b,c) \geqslant u(a) v(b)/2$ given $v(b) \geqslant u(a)$; this is enough to conclude \eqref{eq:claim}.

\medskip\noindent\textbf{Step 2}. Once we can estimate $J(a,b,c)$ from below, we can do $c$-integration to estimate $I(f,g)$. Since $J(a,b,c) \geqslant 0$ for any $a,b,c > 0$, it follows from our claim \eqref{eq:claim} and the estimate \eqref{eq:step1} that
\begin{equation}\label{eq:step2}
\begin{split}
I(f,g)&\geqslant \int_0^\infty \int_0^\infty \bigg(\lambda \int_0^{(\max\{u(a), v(b)\}/2\omega_n)^{1/n}} c^{\lambda-1} J(a,b,c) dc \bigg) da db \\
& \geqslant \frac{(2\omega_n)^{-\lambda/n}}2 \int_0^\infty \int_0^\infty u(a) v(b) (\max\{u(a), v(b)\})^{\lambda /n} da db.
\end{split}
\end{equation}
Next, we split the integral $\int_0^\infty$ evaluated with respect to the variable $b$ in \eqref{eq:step2} into two integrals as follows $\int_0^\infty = \int_0^{a^{p/r}} + \int_{a^{p/r}}^\infty $. Then
\begin{equation}\label{eq:step3}
\begin{split}
\int_0^\infty & \int_0^\infty  u(a) v(b)  (\max\{u(a), v(b)\} )^{\lambda/ n} da db \\
& \geqslant \int_0^\infty u(a)\int_0^{a^{p/r}} v(b)^{1+\lambda/ n} db da + \int_0^\infty u(a)^{1+\lambda/n} \int_{a^{p/r}}^\infty v(b) db da \\
&=\int_0^\infty u(a)\int_0^{a^{p/r}} v(b)^{1+\lambda/ n} db da + \int_0^\infty v(b) \int_0^{b^{r/p}} u(a)^{1 + \lambda/n} da db\\
& = I + II.
\end{split}
\end{equation}
(Note that to obtain \eqref{eq:step3}, we have used the following identity
\begin{equation}\label{eqInterestingFormula}
\int_0^\infty \phi (a) \int_{a^{p/r}}^\infty \psi (b) db da = \int_0^\infty \psi (b) \int_0^{b^{r/p}} \phi (a) da db
\end{equation}
for arbitrary functions $\phi$ and $\psi$; see \cite[Eq. (20), page 110]{liebloss2001}.)

\medskip\noindent\textbf{Step 3}. We now estimate $I$ and $II$ term by term. To estimate $I$, we make use of the reversed H\"older inequality for parameters $n/(n+\lambda)$ and $-n/\lambda$ to obtain
\begin{align*}
\int_0^{a^{p/r}} v(b)^{1+\lambda/n} db &= \int_0^{a^{p/r}} v(b)^{1+\lambda/n} b^{(\lambda +n)(r-1)/n} b^{-(\lambda +n)(r-1)/n} db\notag\\
&\geqslant \bigg(\int_0^{a^{p/r}} v(b) b^{r-1} db \bigg)^{1+\lambda/n} \bigg(\int_0^{a^{p/r}} b^{(\lambda +n)(r-1)/\lambda} db \bigg)^{-\lambda/n}\notag.
\end{align*}
Observe that
\[
\int_0^{a^{p/r}} b^{(\lambda +n) (r-1)/\lambda} db = \frac{p}{r(1-p)}\, \frac\lambda n a^{n(1-p)/\lambda},
\]
since $(1+n/\lambda) (r-1) + 1 = r(1/p-1) n/\lambda > 0$. Therefore, we can conclude that
$$\int_0^{a^{p/r}} v(b)^{1+\lambda/n} db \geqslant \bigg(\frac\lambda n \frac{p}{r(1-p)}\bigg)^{-\lambda/n} a^{p-1} \bigg(\int_0^{a^{p/r}} v(b) b^{r-1} db \bigg)^{1+\lambda/n}.$$
Now, we use the normalization $\int_0^\infty p a^{p-1} u(a) da =1$ and the Jensen inequality to get
\begin{equation}\label{eq:int1}
\begin{split}
I =& \int_0^\infty u(a) \int_0^{a^{p/r}} v(b)^{1+\lambda/n} db da \\
 \geqslant &\frac1{(pr)^{1+\frac \lambda n}} \bigg(\frac\lambda n \frac r{1-r}\bigg)^{-\frac \lambda n} \bigg(\int_0^\infty p u(a) a^{p-1}\int_0^{a^{p/r}} r v(b)b^{r-1} db da\bigg)^{1+ \frac \lambda n}.
\end{split}
\end{equation}
By performing the same argument and using \eqref{eqInterestingFormula}, we can bound the term $II$ as follows
\begin{equation}\label{eq:int2}
\begin{split}
II =& \int_0^\infty v(b) \int_0^{b^{r/p}} u(a)^{1 + \lambda /n} da db \\
\geqslant & \frac1{(pr)^{1+\frac \lambda n}} \bigg(\frac\lambda n \frac p{1-p}\bigg)^{- \frac \lambda n} \bigg(\int_0^\infty r v(b) b^{r-1}\int_0^{b^{r/p}} p u(a)a^{p-1} da db\bigg)^{1+ \frac \lambda n} \\
=& \frac1{(pr)^{1+ \frac \lambda n}} \bigg(\frac\lambda n \frac p{1-p}\bigg)^{- \frac \lambda n} \bigg(\int_0^\infty p u(a) a^{p-1}\int_{a^{p/r}}^\infty r v(b)b^{r-1} db da\bigg)^{1+ \frac \lambda n}.
\end{split}
\end{equation}
By setting
\[
C:= \frac1{(pr)^{1+\lambda/n}} \Big(\frac\lambda n\Big)^{-\lambda/n} \bigg(\max\Big\{\frac{r}{1-r}, \frac{p}{1-p}\Big\}\bigg)^{-\lambda/n},
\]
substituting \eqref{eq:int1} and \eqref{eq:int2} into \eqref{eq:step3}, and using the convexity of the function $\phi(t) = t^{1+\lambda /n}$, we obtain
\begin{equation}\label{eq:step4}
\begin{split}
\int_0^\infty \int_0^\infty u(a) v(b) & \big(\max\{u(a), v(b)\}\big)^{\lambda /n} da db \\
\geqslant & C \bigg(\int_0^\infty p u(a) a^{p-1}\int_0^{a^{p/r}} r v(b)b^{r-1} db da\bigg)^{1+\lambda/n} \\
&+ C \bigg(\int_0^\infty p u(a) a^{p-1}\int_{a^{p/r}}^\infty r v(b)b^{r-1} db da\bigg)^{1+\lambda/n} \\
\geqslant & 2^{-\lambda/n}C .
\end{split}
\end{equation}
Combining \eqref{eq:step2} and \eqref{eq:step4} completes our proof of \eqref{eq:RHLS} with the constant
\begin{equation}\label{eqConstantC_npr}
C(n,p,r) = \frac{(2\omega_n)^{-\lambda/n}}{2^{1+\lambda/n}} \, \frac1{(pr)^{1+\lambda/n}} \bigg(\frac\lambda n\bigg)^{-\frac \lambda n} \bigg(\max\bigg\{\frac{r}{1-r}, \frac{p}{1-p}\bigg\}\bigg)^{- \frac \lambda n}.
\end{equation}

\section{Existence of optimal functions for the reversed HLS inequality: Proof of Theorem \ref{Existence}}

Recall that $p$, $r$, and $\lambda >0$ satisfy $1/p + 1/r - \lambda /n = 2$. For simplicity, we denote $q = r/(r-1) < 0$. Given a function $f$ on $\Rset^n$ which vanishes at infinity, its symmetric decreasing rearrangement is denoted by $f^\star$; see \cite{liebloss2001} or \cite{Bur} for the definitions. It is well-known that if $f\in L^p(\Rset^n)$ for $p > 0$, then $f^\star\in L^p(\Rset^n)$ and $\|f\|_{L^p(\Rset^n)} = \|f^\star\|_{L^p(\Rset^n)}$.

To prove Theorem \ref{Existence}, we first establish the following simple lemma which tells us more about the interaction between $f$ and $f^\star$.

\begin{lemma}\label{lemmaexistence}
We have the following claims:
\begin{itemize}
\item [(i)] For any non-negative functions $f,g$ on $\Rset^n$, we have
\begin{equation}\label{eq:reversedRiesz}
\int_{\Rset^n}\int_{\Rset^n} f(x) |x-y|^\lambda g(y) dxdy \geqslant \int_{\Rset^n}\int_{\Rset^n} f^\star(x) |x-y|^\lambda g^\star(y) dxdy.
\end{equation}
with equality if and only if
\[
f(x) = f^\star(x+x_0), \quad g(x) = g^\star(x+x_0)
\]
for some $x_0\in \Rset^n$.
\item [(ii)] The function $I_\lambda f^\star$ is radially symmetric and strictly increasing.
\item [(iii)] For any non-negative function $f\in L^p(\Rset^n)$, there holds
\begin{equation}\label{eq:decreasenorm}
\|I_\lambda f\|_{L^q(\Rset^n)} \geqslant \|I_\lambda f^\star \|_{L^q(\Rset^n)},
\end{equation}
with equality if and only if $f^\star$ is a strictly decreasing and 
\[
f(x) = f^\star(x +x_0)
\]
for some $x_0\in \Rset^n$.
\end{itemize}
\end{lemma}

\begin{proof}
Inequality \eqref{eq:reversedRiesz} was proven in \cite[Proof of Proposition $9$]{BrascampLieb1976}. For the equality case, we can repeat the proof of the equality case in the Riesz inequality with a remark that the function $t\to t^\lambda$ is strictly increasing; see also \cite[Proof of Theorem 2.10]{Bur}. This completes the proof of (i). It is clear that the function $I_\lambda f^\star$ is radially symmetric. The strictly increasing monotonicity of $I_\lambda f^\star$ and (iii) can be derived from (i) by choosing suitable test functions. 
\end{proof}

We are now in a position to prove Theorem \ref{Existence}. Since the radial symmetry and strictly decreasing of minimizers for \eqref{eq:variationalprob} immediately follow from Lemma \ref{lemmaexistence}, it suffices to prove the existence of a minimizer for \eqref{eq:variationalprob}. For clarity, we divide our proof into several steps.

\medskip\noindent\textbf{Step 1}. \textit{Select a suitable minimizing sequence for \eqref{eq:variationalprob}.} 

We start our proof by letting $\{f_j\}_j$ be a minimizing sequence for \eqref{eq:variationalprob}, so is the sequence $\{f_j^\star \}_j$. Hence, without loss of generality, we can assume at the beginning that $\{f_j\}_j$ is non-negative, radially symmetric, non-increasing sequence. To avoid introducing more notations, we shall write $f_j(x)$ by $f_j(|x|)$. Under this convention and that $\|f_j\|_{L^p(\Rset^n)}=1$, we have
\[
\begin{split}
1 = & n\omega_n \int_0^\infty f_j(r)^p r^{n-1} dr  \geqslant n\omega_n \int_0^R f_j(r)^p r^{n-1} \geqslant \omega_n f_j(R)^p R^n
\end{split}
\]
for any $R > 0$. From this, we obtain the estimate
\[
0\leqslant f_j(r) \leqslant C r^{-n/p}
\]
for any $r > 0$ and for some constant $C$ independent of $j$. In order to go further, we need the following lemma whose proof is similar to that of Lemma 3.2 proven in \cite{dz2014}; see also \cite[Lemma 2.4]{l1983}.

\begin{lemma}\label{farawayzero}
Suppose that $f\in L^p(\Rset^n)$ is non-negative, radially symmetric, and $f(|x|) \leqslant \epsilon |x|^{-n/p}$ for all $|x| > 0$. Then, there exists a constant $C_1 > 0$ independent of $f$ and $\epsilon$ such that 
\begin{equation}\label{eq:farawayzero}
\|I_\lambda f\|_{L^q(\Rset^n)} \geqslant C_1 \epsilon^{1 -p/p_1} \|f\|_{L^p(\Rset^n)}^{p/p_1}.
\end{equation}
for any $p_1\in (0, 2n/(2n+\lambda))$.
\end{lemma} 

\medskip\noindent\textbf{Step 2}. \textit{Existence of a potential minimizer $f_0$ for \eqref{eq:variationalprob}.} 

Set
\[
a_j = \sup_{r > 0} r^{n/p} f_j(r) \in [0,C].
\]
Keep in mind that $\|f_j\|_{L^p(\Rset^n)} = 1$ and $\|I_\lambda f_j\|_{L^q(\Rset^n)} \to \Cscr_{n,p,\lambda} < \infty$. It follows from Lemma \ref{farawayzero} that $a_j \geqslant 2c_0$ for some $c_0 >0$. For each $j$, we choose $\lambda_j > 0$ in such a way that $\lambda_j^{n/p} f_j(\lambda_j) > c_0$. Then we set 
\[
g_j(x) = \lambda_j^{n/p} f_j(\lambda_j x).
\] 
Now, it is routine to check that $\{g_j\}_j$ is also a minimizing sequence for \eqref{eq:variationalprob}. Furthermore, $g_j(1) > c_0$ for any $j$ by our choice for $\lambda_j$. Consequently, we can further assume that the sequence $\{f_j\}_j$ has $f_j(1) > c_0$ for any $j$; otherwise, we can replace the sequence $\{f_j\}_j$ by the sequence $\{g_j\}$, if necessary. 

Similar to Lieb's argument which was based on the Helly theorem, a subsequence of $\{f_j\}_j$ converges weakly to $f_0$ a.e. in $\Rset^n$. It is evident that $f_0$ is non-negative, radially symmetric, non-increasing and is in $L^p(\Rset)$. The rest of our arguments will be used to show that $f_0$ is the desired minimizer for \eqref{eq:variationalprob}.

By Lemma \ref{lemmaexistence}, the function $I_\lambda f_j$ is radially symmetric and strictly decreasing for any $j$. Moreover, for all $x\in \Rset^n$, there holds
\begin{equation}\label{eq:lowerbound}
I_\lambda f_j(x) \geqslant c_0\int_{|y| \leqslant 1} |x-y|^\lambda dy \geqslant C_2 (1 + |x|^\lambda)
\end{equation}
for some new constant $C_2$ independent of $j$. 

\medskip\noindent\textbf{Step 3}. \textit{The function $f_0$ is a minimizer for \eqref{eq:variationalprob}: Preliminaries.} 

Since $\|I_\lambda f_j\|_q$ has the limit $\Cscr_{n,p,\lambda}$, there exists some constant $C_3 > 0$ such that $\|I_\lambda f_j\|_q^q \leqslant C_3$ for any $j$. Therefore
\[
\begin{split}
C_3 \geqslant &\int_{\Rset^n} (I_\lambda f_j(x))^q dx \geqslant  \int_{|x|\leqslant R}(I_\lambda f_j(x))^q dx \geqslant \omega_n I_\lambda f_j(R)^q R^n
\end{split}
\]
for any $R> 0$. Consequently, for all $r>0$, there holds
\[
0 \leqslant (I_\lambda f_j(r))^{-1} \leqslant C_4 r^{n/q}
\]
for some new constant $C_4$ independent of $f_j$. Since $(I_\lambda f_j)^{-1}$ is radially symmetric and non-increasing, it is easy to verify that a subsequence of $\{(I_\lambda f_j)^{-1}\}_j$ converges to $k$ a.e. in $\Rset^n$ for some function $k$. By \eqref{eq:lowerbound} and the dominated convergence theorem, we arrive at
\begin{equation}\label{eqIntegralK=C}
\int_{\Rset^n} k(x)^{-q} dx = \Cscr_{n,p,\lambda}^q.
\end{equation}

\noindent\textbf{Step 4}. \textit{The function $f_0$ is a minimizer for \eqref{eq:variationalprob}: Completed.} 

To realize that $f_0$ is a minimizer for \eqref{eq:variationalprob}, we first prove that $\|f_0\|_{L^p(\Rset^n)} =1$. For this purpose, one could show that $f_j \to f_0$ strongly in $L^p(\Rset^n)$ by employing the rough reversed HLS inequality \eqref{eq:RHLS}. However, it is difficult to adopt this strategy since we cannot control the sign of $f_j - f_0$, which is required when applying \eqref{eq:RHLS}; see \cite[page 17]{dz2014}. In order to avoid such difficulty, we propose an alternative approach. First, we observe the relation \eqref{eqIntegralK=C} to see that the set $\{x : 0< k(x) <\infty\}$ has a positive measure. Therefore, we can choose two distinct points $x_1$ and $x_2$ such that
\[
\lim_{j\to\infty} (I_\lambda f_j) (x_i) = k(x_i)^{-1} 
\]
for $i=1,2$. (This is because $(I_\lambda f_j)^{-1} \to k$ a.e. in $\Rset^n$.) Then, there exists some constant $C_5 > 0$ such that 
\[
I_\lambda f_j(x_i) \leqslant C_5
\]
for $i=1,2$ and for all $j \geqslant 1$. Using the elementary inequality $|x+y|^\lambda \leqslant \max\{1,2^{\lambda-1}\} (|x|^\lambda + |y|^\lambda)$ for any $x,y\in\Rset^n$, we estimate
\begin{equation*}
\begin{split}
\frac{|x_1-x_2|^\lambda}{\max\{1,2^{\lambda-1}\}} \int_{\Rset^n}f_j(y)dy &\leqslant \int_{\Rset^n} |x_1-y|^\lambda f_j(y) dy + \int_{\Rset^n} |x_2-y|^\lambda f_j(y) dy \\
&= I_\lambda f_j(x_1) + I_\lambda f_j(x_2) \leqslant 2 C_5.
\end{split}
\end{equation*}
Thus, there exists another constant $C_6> 0$ such that 
\begin{equation*}
\int_{\Rset^n}f_j(y)dy \leqslant C_6
\end{equation*}
for all $j\geqslant 1$. On one hand, there holds $|x_1-y| \geqslant |y|/3$ for any $R > 2|x_1|$ and any $y$ in the region $\{3R/4 \leqslant |y|\leqslant R\}$. Therefore, by a simple change of variables, we obtain
\[
\begin{split}
C_5 \geqslant &\int_{\{3R/4\leqslant |y|\leqslant R\}} |x_1-y|^\lambda f_j(y) dy \geqslant  3^{-\lambda} f_j(R) R^{n+\lambda} \int_{\{3/4\leqslant |y|\leqslant 1\}} |y|^\lambda dy.
\end{split}
\]
(Note that in the preceding estimate, we have used the fact that $f_j$ is radially symmetric and non-increasing.) Hence, there exists some new constant $C_7 > 0$ such that $f_j(r) \leqslant C_7 r^{-n-\lambda}$ for any $r > 2 |x_1|$ and for all $j\geqslant 1$. Making use of the above estimate, we deduce that
\begin{equation}\label{eq:outsideball}
\begin{split}
\int_{\{|y| > R\}} f_j(x)^p dx \leqslant & C_7^p \int_{\{|y| > R\}}|x|^{-p(n+\lambda)}dx = -\frac{\omega_n q} p C_7^pR^{np/q}.
\end{split}
\end{equation}
Since $\int_{\Rset^n}f_j(y)dy \leqslant C_6$, we also have
\begin{equation}\label{eq:fjgeqR}
\int_{\{f_j > R\}} f_j(x)^p dx \leqslant R^{p-1} \int_{\Rset^n} f_j(x) dx \leqslant C_6 R^{p-1}.
\end{equation}
In view of \eqref{eq:outsideball} and \eqref{eq:fjgeqR}, given $\epsilon >0$, we can select $R > 2|x_1|$ sufficiently large such that 
\[
\int_{\{|y| > R\}} f_j(x)^p dx < \frac\epsilon 2 \quad \text{ and } \quad \int_{\{f_j > R\}} f_j(x)^p dx < \frac\epsilon 2.
\]
We now set $g_j(x) = \min\{f_j(x),R\}$ for each $j\geqslant 1$. By using $\int_{\Rset^n} f_j(x)^p dx =1$, we have
\begin{equation*}
\begin{split}
\int_{\{|y|\leqslant R\}}g_j(x)^p dx& \geqslant \int_{\{|y|\leqslant R\}\cap \{f_j \leqslant R\}} f_j(x)^p dx\\
&=1 - \int_{\{|y|\leqslant R\}\cap \{f_j > R\}} f_j(x)^p dx -\int_{\{|y| > R\}} f_j(x)^p dx\\
&\geqslant 1-\epsilon.
\end{split}
\end{equation*}
For each $R$ fixed, the dominated convergence theorem guarantees that
\begin{equation*}
\lim_{j\to\infty} \int_{\{|y|\leqslant R\}}g_j(x)^p dx = \int_{\{|y|\leqslant R\}}\left(\min\{f_0(x),R\}\right)^p dx.
\end{equation*}
Therefore, as $R \to +\infty$, we arrive at
\[
\int_{\Rset^n}f_0(x)^p dx \geqslant 1-\epsilon,
\]
for any $\epsilon >0$. From this we conclude $\int_{\Rset^n}f_0(x)^p dx \geqslant 1$. On the other hand, we have $\int_{\Rset^n}f_0(x)^p dx \leqslant 1$ by the Fatou lemma. This means that $\|f_0\|_{L^p(\Rset^n)} =1$.

To prove that $f_0$ is a minimizer for \eqref{eq:variationalprob}, we apply the Fatou lemma again to get
\[
k(x) = \lim_{j\to\infty} (I_\lambda f_j(x))^{-1} =(\lim_{j\to\infty}I_\lambda f_j(x))^{-1} \leqslant (I_\lambda f_0(x))^{-1}
\]
for a.e. $x$ in $\Rset^n$. Combining the preceding estimate and \eqref{eqIntegralK=C} gives
\[
\Cscr_{n,p,\lambda} = \Cscr_{n,p,\lambda}\|f_0\|_{L^p(\Rset^n)} \leqslant \|I_\lambda f_0\|_{L^q(\Rset^n)} \leqslant \left(\int_{\Rset^n} k(x)^{-q} dx\right)^{1/q} = \Cscr_{n,p,\lambda}.
\] 
This shows that $f_0$ is indeed a minimizer for \eqref{eq:variationalprob}.

\section{Classification of non-negative, measurable solutions of \eqref{eqIntegralSystem}: Proof of Theorems \ref{thmOPTIMAL} and \ref{thmCLASSIFICATION}}
\label{secCLASSIFICATION}

Before proving Theorem \ref{thmCLASSIFICATION}, it is necessary to mention the relation between the optimizers for \eqref{eq:RHLS-diagonal} and the system \eqref{eq:Euleur}. Since the argument is simple, we include it below to make this paper self-contained.

To see how optimizers for \eqref{eq:RHLS-diagonal} and \eqref{eq:Euleur} are related to each other, let us first denote
\[F_\lambda (f,g) = \int_{\Rset^n} \int_{\Rset^n} f(x) |x-y|^\lambda g(y) dx dy.\]
Then, to compute the sharp constant $\Cscr_{n,\lambda}$ it is necessary to minimize the functional $F_\lambda$ along with the following two constraints
\[
\int_{\Rset^n} |f(x)|^{2n/(2n+\lambda)} dx =1 \quad \text{ and } \quad \int_{\Rset^n} |g(x)|^{2n/(2n+\lambda)} dx =1.
\] 
By a fairy simple calculation, the first variation of the functional $F_\lambda$ with respect to $f$ is 
\[D_f (F_\lambda) (f,g) (h) = \int_{\Rset^n} \bigg( \int_{\Rset^n} |x-y|^\lambda g(y) dy \bigg) h(x) dx\]
while the first variation of the constraint $\int_{\Rset^n} |f(x)|^{2n/(2n+\lambda)} dx =1$ with respect to $f$ is
\[
\frac{2n}{2n+\lambda} \int_{\Rset^n} |f(x)|^{-\lambda/(2n+\lambda)-1} f(x) h(x) dx.
\]
Therefore, by the Lagrange multiplier theorem, there exists some constant $\alpha$ such that
\[
\int_{\Rset^n} \bigg( \int_{\Rset^n} |x-y|^\lambda g(y) dy \bigg) h(x) dx = \alpha \int_{\Rset^n} |f(x)|^{-\lambda/(2n+\lambda)-1} f(x) h(x) dx
\]
holds for all $h$. Consequently, both $f$ and $g$ must satisfy
\[
\alpha |f(x)|^{-\lambda/(2n+\lambda)-1} f(x) = \int_{\Rset^n} |x-y|^\lambda g(y) dy.
\]
Interchanging $f$ and $g$, we conclude that $f$ and $g$ must also satisfy
\[
\beta |g(x)|^{-\lambda/(2n+\lambda)-1} g(x) = \int_{\Rset^n} |x-y|^\lambda f(y) dy
\]
for some new constant $\beta$. Note that the balance condition guarantees that $\alpha = \beta = 1/F_\lambda (f,g)$. Hence, up to a constant multiple, the relation above leads us to \eqref{eq:Euleur}. From this, it suffices to classify positive solutions of \eqref{eq:systemDouZhu} in order to understand the structure of optimizers for \eqref{eq:RHLS-diagonal}.

\subsection{Preliminaries}

In this subsection, we setup some preliminary findings necessary for the rest of our analysis. The most important part of this section is to obtain \textit{a prior} estimates for solutions of \eqref{eqIntegralSystem}; see Lemma \ref{lem-Growth} below. Here and in what follows, by $\lesssim$ and $\gtrsim$ we mean inequalities up to $p$, $q$, and dimensional constants.

\begin{lemma}\label{lem-Growth}
Given $n \geqslant 1$ and $p,q>0$, let $(u,v)$ be a pair of non-negative Lebesgue measurable functions in $\Rset^n$ satisfying \eqref{eqIntegralSystem}. Then
\begin{equation}\label{eqGrowth1}
\int_{\Rset^n} {(1 + |y|^p )u{{(y)}^{ - q}}dy} < \infty , \quad \int_{\Rset^n} {(1 + |y|^p )v{{(y)}^{ - q}}dy} < \infty ,
\end{equation}
and
\begin{equation}\label{eqGrowth2}
\begin{split}
\mathop {\lim }\limits_{|x| \to \infty } \frac{{u(x)}}{{|x|^p }} =& {\int_{\Rset^n} {v{{(y)}^{ - q}}dy} } , \quad \mathop {\lim }\limits_{|x| \to \infty } \frac{{v(x)}}{{|x|^p }} = {\int_{\Rset^n} {u{{(y)}^{ - q}}dy} } ,
\end{split}
\end{equation}
and $u$ and $v$ are bounded from below in the following sense 
\begin{equation}\label{eqGrowth3}
u(x),v(x) \gtrsim 1 + |x|^p
\end{equation}
and above in the following sense
\begin{equation}\label{eqGrowth4}
u(x),v(x) \lesssim 1 + |x|^p
\end{equation}
for all $x \in \Rset^n$. In other words, there holds
\[
 \frac{1 + |x|^p }{C} \leqslant u(x), v(x) \leqslant C(1 + |x|^p )
\]
in $\Rset^n$ for some constant $C \geqslant 1$.
\end{lemma}

\begin{proof}
We begin by noting from \eqref{eqIntegralSystem} that both $u$ and $v$ are strictly positive everywhere in $\Rset^n$ and are finite within a set of positive measure. Consequently, there exist some large constant $R>1$ and some Lebesgue measurable set $E \subset \Rset^n$ such that
\begin{equation}\label{eqSetE}
E \subset \{ y : u(y) < R, v(y) < R\} \cap B(0,R)
\end{equation}
with $\text{meas}(E) \geqslant 1/R$. Using this, we can easily bound $v$ from below as follows
\[
\begin{split}
v(x) \geqslant &\int_E {|x - y|^p u(y)^{ - q}dy}  \geqslant  \frac{1}{R^q}\int_E {|x - y|^p dy} =\frac{1}{R^q}\int_{E+x} {|y|^p dy} 
\end{split}
\]
for any $x \in \Rset^n$. Choose $\varepsilon > 0$ small enough and then fix it in such a way that $\vol (B(0,\varepsilon)) < |E| /2$. Then we can estimate
\[
\begin{split}
\int_{E+x} {|y|^p dy} & \geqslant \int_{E+x \backslash B(0,\varepsilon)} {|y|^p dy}  \geqslant \varepsilon^p \int_{E+x \backslash B(0,\varepsilon)} dy \\
&  = \varepsilon^p \big( |E+x| -\vol (B(0,\varepsilon)) \big).
\end{split}
\]
From this, it is clear that $v$ is bounded from below by some positive constant. The same reasoning can be applied to $u$. This shows that there exists some constant $C_0>0$ such that 
\begin{equation}\label{eqUVBoundedFromBelow}
u(x), v(x) > C_0
\end{equation}
everywhere in $\Rset^n$. 

\medskip\noindent\textbf{Proof of \eqref{eqGrowth3}}. To improve the bound of $u$ and $v$ in \eqref{eqUVBoundedFromBelow}, we first consider the region $\{|x| \geqslant 2R\}$ where $R$ is given in \eqref{eqSetE}. Note that for every $y \in E \subset B(0,R)$, there holds $|x - y| \geqslant |x| - |y| \geqslant |x|/2$ since $|x| \geqslant 2R$. Therefore
\[
v(x) \geqslant \frac{1}{{{R^q}}}\int_E {|x - y|^p dy} \geqslant \frac{\text{vol}(E)}{(2R)^p} |x|^p
\]
for any $|x| \geqslant 2R$. A similar argument shows that $u(x) \geqslant \text{vol}(E) (2R)^{-p}|x|^p $ in the region $\{ |x| \geqslant 2R\}$. Hence, it is easy to select some large constant $C>1$ such that \eqref{eqGrowth3} holds in the region $\{ |x| \geqslant 2R\}$. Using \eqref{eqUVBoundedFromBelow}, we can decrease $C$, if necessary, to obtain \eqref{eqGrowth3} in the ball $\{ |x| \leqslant 2R\}$.

\medskip\noindent\textbf{Proof of \eqref{eqGrowth1}}. We will only estimate $v$ since $u$ can be estimated in a similar manner. For this purpose, we will first show that $u^{-q} \in L^1(\Rset^n)$. For some $\overline x$ satisfying $1 \leqslant |\overline x| \leqslant 2$, it is clear that
\[\int_{\Rset^n} {|\overline x - y|^p u{{(y)}^{ - q}}dy} = v(\overline x )<+ \infty .\]
Observe that for any $y \in \Rset^n \backslash B(0,4)$, $|\overline x - y | \geqslant |y| -|\overline x | > 1$. Hence
\[
\int_{\Rset^n \backslash B(0,4) } u(y)^{ - q}dy < \int_{\Rset^n} {|\overline x - y|^p u{{(y)}^{ - q}}dy} < +\infty.
\]
In the small ball $B(0,4)$, we notice that
\[
\int_{ B(0,4) } u(y)^{ - q}dy \lesssim \int_{B(0,4)} (1+|y|^p))^{-q} dy < +\infty.
\]
Thus, $u^{-q} \in L^1(\Rset^n)$ as claimed. To conclude \eqref{eqGrowth1}, it suffices to prove that
\begin{equation}\label{eqGrowth1-suffice}
\int_{\Rset^n} |y|^p u(y)^{ - q}dy < +\infty.
\end{equation}
To see \eqref{eqGrowth1-suffice}, we observe that $|y| \leqslant 2|\overline x - y| $ for all $y \in \Rset^n \backslash B(0,4) $. Therefore,
\[
\int_{\Rset^n \backslash B(0,4) } |y|^p u(y)^{ - q}dy\lesssim \int_{\Rset^n \backslash B(0,4)} {|\overline x - y|^p u{{(y)}^{ - q}}dy} < +\infty.
\]
In the small ball $B(0,4)$, it is apparent that
\[
\int_{ B(0,4) } |y|^p u(y)^{ - q}dy \lesssim \int_{B(0,4)}u(y)^{-q} dy < +\infty,
\]
since $u^{-q} \in L^1(\Rset^n)$. Thus \eqref{eqGrowth1-suffice} follows and so does \eqref{eqGrowth1}. 

\medskip\noindent\textbf{Proof of \eqref{eqGrowth2}}. We will only consider the limit $|x|^{-p} v(x)$ as $|x| \to \infty$ since the limit $|x|^{-p} u(x)$ can be proven similarly. Using \eqref{eqIntegralSystem}, we obtain
\begin{equation}\label{eqProofGrowth2-1}
\begin{split}
\mathop {\lim }\limits_{|x| \to \infty } \frac{{v(x)}}{{|x|^p }} =& \mathop {\lim }\limits_{|x| \to \infty } {\int_{\Rset^n} {\frac{{|x - y|^p }}{{|x|^p}} u{{(y)}^{ - q}}dy} }.
\end{split}
\end{equation}
Observe that as $|x| \to +\infty$, $( |x - y|/|x| )^p u{(y)^{ - q}} \to u{(y)^{ - q}}$ almost everywhere $y$ in $\Rset^n$. Hence we can apply the Lebesgue dominated convergence theorem to pass \eqref{eqProofGrowth2-1} to the limit to conclude \eqref{eqGrowth1}, provided we can show that $|x - y|^p |x{|^{ - p}}u{(y)^{ - q}}$ is bounded by some integrable function. To this end, we observe that $|x-y|^p \leqslant (|x|+|y|)^p \lesssim  (|x|^p +|y|^p)$. Hence, if $|x| > 1$ then
\[ \Big( \frac{ |x - y|} {|x|} \Big)^p u (y)^{ - q}  \lesssim (1+|y|^p)u{(y)^{ - q}}.\]
Our proof now follows by observing $ (1 + |y|^p )u (y)^{ - q}  \in L^1(\Rset^n)$ by \eqref{eqGrowth1}.

\medskip\noindent\textbf{Proof of \eqref{eqGrowth4}}. We now observe \eqref{eqGrowth2} to see that there exists some large number $k>1/R$ such that 
\[
\frac{u(x)}{|x|^p} < 1 + \int_{\Rset^n} v(y)^{-q} dy
\]
in $\Rset^n \backslash B(0,kR)$. In the ball $B(0,kR)$, we can easy to estimate $|x-y|^p \lesssim |x|^p + |y|^p$. This will help us to conclude that
\[
u(x) \lesssim (kR)^p \int_{\Rset^n} (1+|y|^p) v(y)^{-q} dy
\]
in the ball $B(0,kR)$. Using the preceding inequality and our estimate for $u$ outside $B(0,kR)$, we obtain the desired estimate. Similarly, our estimate for $v$ follows.  
\end{proof}

In the next result, we will prove a regularity result similar to \cite[Lemma 5.2]{l2004}.

\begin{lemma}\label{lem-Regularity}
For $n \geqslant 1$ and $p,q>0$, let $(u,v)$ be a pair of non-negative Lebesgue measurable functions in $\Rset^n$ satisfying \eqref{eqIntegralSystem}. Then $u$ and $v$ are smooth.
\end{lemma}

\begin{proof}
Our proof is similar to \cite[Lemma 5.2]{l2004}. Let $R>0$ be arbitrary, we decompose $u$ and $v$ into the following way
\[
\left\{
\begin{split}
 u(x) =& u_R^1(x) + u_R^2(x) = \bigg(\int_{|y| \leqslant 2R} + \int_{|y| > 2R} \bigg)|x - y|^p v{(y)^{ - q}}dy, \\
 v(x) =& v_R^1(x) + v_R^2(x) = \bigg(\int_{|y| \leqslant 2R} + \int_{|y| > 2R} \bigg)|x - y|^p u{(y)^{ - q}}dy. \\ 
\end{split} 
\right.
\]
Using \eqref{eqGrowth1}, we can continuously differentiate $u_R^2$ and $v_R^2$ under the integral sign for any $|x|<R$. Consequently, $u_R^2 \in C^\infty (B(0,R))$ and $v_R^2 \in C^\infty (B(0,R))$. In view of \eqref{eqGrowth3} and \eqref{eqGrowth4}, we know that $u^{-q} \in L^\infty (B(0, 2R))$ which implies that $v_R^1$ is at least H\"older continuous in $B(0,R)$. Similarly, the same argument shows that $u_R^1$ is also at least H\"older continuous in $B(0,R)$. Hence, we have just proven that $u$ and $v$ are at least H\"older continuous in $B(0,R)$. This means that $u$ and $v$ are at least H\"older continuous in the whole space $\Rset^n$ since $R>0$ is arbitrary. A standard bootstrap argument shows that $u \in C^\infty (\Rset^n)$ and similarly $v \in C^\infty (\Rset^n)$.
\end{proof}

Once we obtain the smoothness property for solutions of \eqref{eqIntegralSystem}, we can narrow the range for $q$ as follows.

\begin{proposition}\label{pro-RangeForQ}
For $n \geqslant 1$ and $p,q>0$, it is necessary to have $q \leqslant 1+2n/p$.
\end{proposition}

\begin{proof}
We suppose by contradiction that $q > 1+2n/p$. Then the proof is a direct consequence of \cite[Theorem 1]{hy2013} and Lemma \ref{lem-Regularity}; see also \cite[Theorem 1.1]{lei2015}; hence we omit its details.
\end{proof}

We note that the statement in \cite[Theorem 1]{hy2013} is rather general as it already concludes the necessary condition for \eqref{eqIntegralSystem} to have solutions, which is when $q = 1+2n/p$. Unfortunately, it is not clear for us to check whether \cite[Eq. (20)]{hy2013} holds; hence we cannot exclude the possibility of $q < 1+2n/p$. The only argument that the authors gave to support \cite[Eq. (20)]{hy2013} is to follow the argument in \cite[Proof of Theorem 1]{xu2007}. Fortunately, a simple computation shows that such an argument works provided $q > 1+2n/p$, which is coincidentally our contradiction assumption; therefore this is sufficient for us to conclude the proof above. 

In the same spirit, we also want to mention that \cite[Theorem 1.1]{lei2015} concludes $q = 1+2n/p$ provided $q> 1+n/p$. Hence, we cannot directly conclude $q = 1+2n/p$ without providing certain conditions for $p$ and $q$. However, under our contradiction assumption, it is safe to make use of either \cite[Theorem 1]{hy2013} or \cite[Theorem 1.1]{lei2015} to narrow the range of $q$ as we have just done above.

We also note, after establishing the condition $q \leqslant 1+2n/p$, that eventually we shall see that $q = 1+2n/p$. In view of the compatible condition $1/p + 1/r -\lambda /n =2$, it is rigorous to see that the condition $q = 1+2n/p$ follows from the condition $1/p + 1/r -\lambda /n =2$ if we set $p=r$ and $p=\lambda=q$.

\subsection{The method of moving spheres for systems}

As a consequence of Proposition \ref{pro-RangeForQ}, from now on, we will only consider the case $q \leqslant 1+2n/p$. Let $w$ be a positive function on $\Rset^n$. For $x \in \Rset^n$ and $\lambda>0$ we define
\begin{equation}\label{eqFunctionChange}
{w_{x,\lambda }}( \xi ) =  \Big( \frac{|\xi-x|}{\lambda } \Big)^p w({ \xi^{x,\lambda }})
\end{equation}
for all $\xi \in \Rset^n$ where
\begin{equation}\label{eqVariableChange}
{\xi^{x,\lambda }} = x + {\lambda ^2}\frac{{\xi - x}}{{{{\left| {\xi - x} \right|}^2}}}.
\end{equation}
By changing the variable $y=z^{x,\lambda}$, we have
\begin{equation}\label{eqJacobian}
dy=\Big( \frac{\lambda}{|z-x|} \Big)^{2n} dz.
\end{equation}
Note that if $y=z^{x,\lambda}$, then $z=y^{x,\lambda}$. Therefore,
\[\begin{split}
\int_{|y-x| \geqslant \lambda} |\xi^{x,\lambda} - y|^p v(y)^{-q} dy =& \int_{|z-x| \leqslant \lambda} \frac{ |\xi^{x,\lambda} - z^{x,\lambda}|^p }{ v(z^{x,\lambda})^{q} } \Big( \frac{\lambda}{|z-x|} \Big)^{2n} dz \\
=& \int_{|z-x| \leqslant \lambda}  \frac{   |\xi^{x,\lambda} - z^{x,\lambda}|^p  }{  v_{x,\lambda} (z)^q } \Big( \frac{\lambda}{|z-x|} \Big)^{2n-pq}  dz.
\end{split}\]
Using the relation $|z - x| |\xi - x| |{\xi ^{x,\lambda }} - {z^{x,\lambda }}| = \lambda^2 |\xi - z|$, we obtain
\[\begin{split}
{\Big( {\frac{\lambda } {|\xi - x|}} \Big)^{ - p}} & \int_{|y - x| \geqslant \lambda } | {\xi ^{x,\lambda }} - y|^p v{(y)^{ - q}}dy\\
 =& \int_{|y - x| \geqslant \lambda } {{{\left( {\frac{\lambda } {|z - x|} \frac{{|{\xi ^{x,\lambda }} - {z^{x,\lambda }}|}}{{|\xi - z|}}} \right)}^{ - p}}|} {\xi ^{x,\lambda }} - y|^p v{(y)^{ - q}}dy\\
 =& \int_{|z - x| \leqslant \lambda } | \xi - z|^p {\Big( {\frac{\lambda } {|z - x|} } \Big)^{2n - pq + p}} v_{x,\lambda}(z)^{ - q}dz.
\end{split}\]
Similarly,
\[\begin{split}
{\Big( {\frac{\lambda } {|\xi - x|}} \Big)^{ - p}} & \int_{|y - x| \leqslant \lambda } | {\xi ^{x,\lambda }} - y|^p v{(y)^{ - q}}dy\\
 =& \int_{|z - x| \geqslant \lambda } | \xi - z|^p {\Big( {\frac{\lambda } {|z - x|} } \Big)^{2n - pq + p}} v_{x,\lambda}(z)^{ - q}dz.
\end{split}\]

\begin{lemma}\label{lem0}
For any solution $(u,v)$ of \eqref{eqIntegralSystem}, we have
\[{u_{x,\lambda }}(\xi ) = \int_{\Rset^n} | \xi - z|^p {\Big( {\frac{\lambda } {|z - x|} } \Big)^{2n - pq + p}}{v_{x,\lambda }}{(z)^{ - q}}dz\]
and
\[{v_{x,\lambda }}(\xi ) = \int_{\Rset^n} | \xi - z|^p {\Big( {\frac{\lambda } {|z - x|} } \Big)^{2n - pq + p}}{u_{x,\lambda }}{(z)^{ - q}}dz \]
for any $\xi \in \Rset^n$.
\end{lemma}

\begin{proof}
Using our system \eqref{eqIntegralSystem}, we obtain
\[\begin{split}
{u_{x,\lambda }}(\xi ) =& {\Big( {\frac {|\xi - x|}{\lambda }} \Big)^p}\int_{\Rset^n} | {\xi ^{x,\lambda }} - y|^p v{(y)^{ - q}}dy\\
 =& \int_{\Rset^n} | \xi - z|^p {\Big( {\frac{\lambda } {|z - x|} } \Big)^{2n - pq + p}}{v_{x,\lambda }}{(z)^{ - q}}dz.
\end{split}\]
The formula for $v$ follows the same line of argument as above.
\end{proof}

\begin{lemma}\label{lem1}
For any solution $(u,v)$ of \eqref{eqIntegralSystem}, we have
\[{u_{x,\lambda }}(\xi ) - u(\xi ) = \int_{|z - x| \geqslant \lambda } k(x,\lambda ;\xi ,z) \bigg[v{(z)^{ - q}} - {\Big( {\frac{\lambda } {|z - x|} } \Big)^{2n - pq + p}}{v_{x,\lambda }}{(z)^{ - q}} \bigg] dz\]
and
\[{v_{x,\lambda }}(\xi ) - v (\xi ) = \int_{|z - x| \geqslant \lambda } k(x,\lambda ;\xi ,z) \bigg[ u{(z)^{ - q}} - {\Big( {\frac{\lambda } {|z - x|} } \Big)^{2n - pq + p}}{u_{x,\lambda }}{(z)^{ - q}} \bigg] dz\]
for any $\xi \in \Rset^n$ where
\[k(x,\lambda ;\xi ,z) = \Big( {\frac {|\xi - x|}{\lambda }} \Big)^p |{\xi ^{x,\lambda }} - z|^p - |\xi - z|^p .\]
Moreover, $k(x,\lambda; \xi,z)>0$ for any $|\xi - x| > \lambda>0$ and $|z - x| > \lambda>0$.
\end{lemma}

\begin{proof}
We observe that
\[\begin{split}
{u_{x,\lambda }}(\xi ) =& \int_{|z - x| \geqslant \lambda } | \xi - z|^p {\Big( {\frac{\lambda } {|z - x|} } \Big)^{2n - pq + p}}{v_{x,\lambda }}{(z)^{ - q}}dz\\
& + {\Big( {\frac {|\xi - x|}{\lambda }} \Big)^p}\int_{|z - x| \geqslant \lambda } | {\xi ^{x,\lambda }} - z|^p v{(z)^{ - q}}dz
\end{split}\]
and that
\[\begin{split}
u(\xi ) =& \int_{|z - x| \geqslant \lambda } | {\xi ^{x,\lambda }} - z|^p v{(z)^{ - q}}dz\\
& + {\Big( {\frac {|\xi - x|}{\lambda }} \Big)^p}\int_{|z - x| \geqslant \lambda } | {\xi ^{x,\lambda }} - z|^p {\Big( {\frac{\lambda } {|z - x|} } \Big)^{2n - pq + p}}{v_{x,\lambda }}{(z)^{ - q}}dz.
\end{split}\]
Therefore,
\[{u_{x,\lambda }}(\xi ) - u(\xi ) = \int_{|z - x| \geqslant \lambda } k(x,\lambda ;\xi ,z) \bigg[v{(z)^{ - q}} - {\Big( {\frac{\lambda } {|z - x|} } \Big)^{2n - pq + p}}{v_{x,\lambda }}{(z)^{ - q}} \bigg] dz,\]
where
\[k(x,\lambda ;\xi ,z) = {\Big( {\frac {|\xi - x|}{\lambda }} \Big)^p}|{\xi ^{x,\lambda }} - z|^p - |\xi - z|^p .\]
The representation of $(v_{x,\lambda }- v)(\xi )$ can be obtain in a similar manner. Finally, the positivity of the kernel $k$ for any $|\xi - x| > \lambda $ and $|z - x| > \lambda $ is apparent using the formula
\[
\Big( {\frac {|\xi - x|}{\lambda }} \Big)^2 |{\xi ^{x,\lambda }} - z|^2 - |\xi - z|^2 = \frac 1{\lambda^2} \big( \lambda^2 - |z-x|^2\big) \big( \lambda^2 - |\xi -x|^2\big).
\]
Thus the proof follows.
\end{proof}

For future use, we note that $|\xi - x| |\xi^{x,\lambda} -z| = |z-x| |z^{x,\lambda} - \xi|$; hence we can rewrite the kernel $k$ as follows
\[k(x,\lambda ;\xi ,z) = {\Big( {\frac {|z - x|} {\lambda }} \Big)^p}| \xi - z^{x,\lambda } |^p - |\xi - z|^p .\]
Therefore, each component of $\nabla_\xi k(x,\lambda ;\xi ,z)$ can be easily calculated as
\begin{equation}\label{eqGradientOfK}
\begin{split}
\partial_{\xi_i} k(x,\lambda ;\xi ,z) = & p {\Big( {\frac {|z - x|} {\lambda }} \Big)^p}|\xi - z^{x,\lambda } |^{p-2} \big(\xi_i - (z^{x,\lambda })_i \big) \\
& - p |\xi - z|^{p-2} (\xi_i - z_i) \\
= & p {\Big( {\frac {|z - x|} {\lambda }} \Big)^p} \bigg(\frac{ |\xi - x|}{ |z-x|} |\xi^{x,\lambda} -z| \bigg) ^{p-2} \big(\xi_i - (z^{x,\lambda })_i \big) \\
& - p |\xi - z|^{p-2} (\xi_i - z_i) .
\end{split}
\end{equation}
In particular, 
\begin{equation}\label{eqGradientOfKAtX}
\begin{split}
\nabla k(x,\lambda ;\xi ,z) \cdot \xi = & p {\frac{ |z - x|^2 |\xi - x|^{p-2} }{\lambda^p } } |\xi^{x,\lambda} -z| ^{p-2} \big(|\xi|^2 - z^{x,\lambda }\cdot \xi \big) \\
& - p |\xi - z|^{p-2} (|\xi|^2 - z \cdot \xi ).
\end{split}
\end{equation}

In the following lemma, we will prove that we can apply the method of moving spheres.

\begin{lemma}\label{lemStartMS}
For each $x \in \Rset^n$, there exists $\lambda_0(x)>0$ such that
\[{u_{x,\lambda }}(y) \geqslant u(y), \quad {v_{x,\lambda }}(y) \geqslant v(y)\]
for any point $y \in \Rset^n$ and any $\lambda$ such that $|y-x| \geqslant \lambda$ with $0 < \lambda< \lambda_0(x)$.
\end{lemma}

\begin{proof}
Since $u$ is a positive $C^1$-function and $p>0$, there exists some $r_0>0$ sufficiently small such that
\[\nabla _y \big( |y-x|^{-p/2} u(y) \big) \cdot (y-x) < 0\]
for all $0<|y-x|<r_0$. Consequently, 
\[
\begin{split}
u_{x,\lambda} (y) =& {\Big( {\frac{{|y-x|}}{\lambda }} \Big)^p}u (y^{x,\lambda}) = |y-x|^{p/2} |y^{x,\lambda}-x|^{-p/2} u (y^{x,\lambda}) >  u(y)
\end{split}
\]
for all $0<\lambda < |y-x| < r_0$. Note that in the previous estimate, we made use of the fact that if $|y-x| > \lambda$, then $|y^{x,\lambda} -x| < \lambda$. For sufficiently small $\lambda_0 \in (0, r_0)$ and for all $0<\lambda<\lambda_0$, we have
\[ 
u_{x, \lambda }(y) \geqslant {\Big( {\frac{{|y-x|}}{\lambda }} \Big)^p}\mathop {\inf }\limits_{B(x, r_0)} u \geqslant u(y)
\]
for all $|y-x| \geqslant r_0$. Hence, we have just shown that $u_{x, \lambda }(y) \geqslant u(y)$ for all point $y \in \Rset^n$ and any $\lambda$ such that $|y-x| \geqslant \lambda$ with $0 < \lambda< \lambda_0$. A similar argument shows that $v_{x, \lambda }(y) \geqslant v(y)$ for all point $y$ and any $\lambda$ such that $|y-x| \geqslant \lambda$ with $0 < \lambda< \lambda_1$ for some $\lambda_1 \in (0, r_1)$. By choosing $\lambda_0 (x) = \min\{ \lambda_0, \lambda_1\}$, we obtain the desired result.
\end{proof}

For each $x \in \Rset^n$ we define
\[
\overline \lambda (x) := \sup \left\{ {\mu > 0:{u_{x,\lambda }}(y) \geqslant u(y),{v_{x,\lambda }}(y) \geqslant v(y), \quad \forall 0 < \lambda < \mu ,\left| {y - x} \right| \geqslant \lambda } \right\}.
\]
From Lemma \ref{lemStartMS} above, we get $0 < \overline \lambda (x) \leqslant + \infty$. In the next few lemmas, we will show that whenever $\lambda (x)$ is finite for some point $x$, we are able to write down $(u,v)$ precisely.

\begin{lemma}\label{lem3}
If $\overline \lambda (x_0) <\infty$ for some point $x_0 \in \Rset^n$, then
\[u_{x_0,\overline \lambda (x_0)} \equiv u, \quad {v_{{x_0},\overline \lambda ({x_0})}} \equiv v\]
in $\Rset^n$. In addition, we obtain $q = 1 + 2n/p$.
\end{lemma}

\begin{proof}
By the definition of $\overline \lambda (x_0)$, we know that
\begin{equation}\label{eqProof0}
{u_{x_0,\overline \lambda (x_0)}}(y) \geqslant u(y), \quad {v_{x_0,\overline \lambda (x_0)}}(y) \geqslant v(y)
\end{equation}
for any $\left| {y - x_0} \right| \geqslant \overline \lambda (x_0)$. From Lemma \ref{lem1}, we obtain
\begin{equation}\label{eqProof1}
\begin{split}
u_{x_0,\overline \lambda (x_0)}(y) - u(&y) \\
= \int_{|z - x_0| \geqslant \overline \lambda (x_0)} & \left\{ \begin{gathered}
  k( x_0,\overline \lambda (x_0);y,z ) \times  \hfill \\
  \bigg[v{(z)^{ - q}} - {\left( {\frac{{\overline \lambda (x_0)}}{{|z - x_0|}}} \right)^{2n - pq + p}}{v_{x_0,\overline \lambda (x_0)}}{(z)^{ - q}}\bigg] \hfill \\ 
\end{gathered}  \right\}   dz
\end{split}
\end{equation}
and
\begin{equation}\label{eqProof2}
\begin{split}
v_{x_0,\overline \lambda (x_0)}(y) - v(&y) \\
= \int_{|z - x_0| \geqslant \overline \lambda (x_0)} &  \left\{ \begin{gathered}
k( x_0,\overline \lambda (x_0);y,z )  \times  \hfill \\ 
\bigg[u{(z)^{ - q}} - {\left( {\frac{{\overline \lambda (x_0)}}{{|z - x_0|}}} \right)^{2n - pq + p}}{u_{x_0,\overline \lambda (x_0)}}{(z)^{ - q}}\bigg] \hfill \\ 
\end{gathered}  \right\}  dz
\end{split}
\end{equation}
for any $y \in \Rset^n$. Keep in mind that $2n - pq + p \geqslant 0$, there are two possible cases:

\medskip\noindent\textbf{Case 1}. Suppose that either ${u_{x_0,\overline \lambda (x_0)}}(y) = u(y)$ or ${v_{x_0,\overline \lambda (x_0)}}(y) = v(y)$ for any $\left| {y - x_0} \right| \geqslant \overline \lambda (x_0)$ occurs. Without loss of generality, we assume that the former case occurs. Using \eqref{eqProof1} and the positivity of the kernel $k$, we get that $2n - pq + p=0$ and that ${v_{x_0,\overline \lambda (x_0)}}(y) = v(y)$ for any $\left| {y - x_0} \right| \geqslant \overline \lambda (x_0)$. Similarly, by \eqref{eqProof1} we conclude ${u_{x_0,\overline \lambda (x_0)}}(y) = u(y)$ in the whole $\Rset^n$. A similar argument also shows that ${v_{x_0,\overline \lambda (x_0)}}(y) = v(y)$ in $\Rset^n$.

\medskip\noindent\textbf{Case 2}. Suppose that $u_{x_0,\overline \lambda (x_0)}(y) > u(y)$ and ${v_{x_0,\overline \lambda (x_0)}}(y) > v(y)$ for any $\left| {y - x_0} \right| \geqslant \overline \lambda (x_0)$. In this case, we will obtain a contradiction by showing that we can slightly move spheres a little bit over $\overline \lambda (x_0)$ which then violates the definition of $\overline \lambda (x_0)$. To reach such a contradiction, we shall prove that there exists some small number $\varepsilon>0$ such that
\[
u_{x,\lambda } (y) \geqslant u(y) \quad \text{ and } \quad v_{x,\lambda }(y) \geqslant v(y)
\]
for all $0 < \lambda < \overline \lambda (x_0) + \varepsilon$ and all $| y - x | \geqslant \lambda$. Indeed, using \eqref{eqProof0} and \eqref{eqProof1}, in the region ${|z - x_0| \geqslant \overline \lambda (x_0)}$, we obtain
\[v{(z)^{ - q}} - {\left( {\frac{{\overline \lambda (x_0)}}{{|z - x_0|}}} \right)^{2n - pq + p}}{v_{x_0,\overline \lambda (x_0)}}{(z)^{ - q}} \geqslant v{(z)^{ - q}} - {v_{x_0,\overline \lambda (x_0)}}{(z)^{ - q}}.\]
Hence, 
\begin{equation}\label{eqProof3}
\begin{split}
{u_{x_0,\overline \lambda (x_0)}}(y) &- u(y) \geqslant \int_{|z - x_0| \geqslant \overline \lambda (x_0)} 
 \left\{ \begin{gathered}
k( x_0,\overline \lambda (x_0);y,z ) \times  \hfill \\ 
\Big [v{(z)^{ - q}} - {v_{x_0,\overline \lambda (x_0)}}{(z)^{ - q}}\Big] \hfill \\ 
\end{gathered}  \right\}  dz.
\end{split}
\end{equation}

\medskip\noindent\textbf{Estimate of $u_{x_0,\lambda } - u$ outside $B(x_0,\overline \lambda (x_0) + 1)$.} Using the Fatou lemma, from \eqref{eqProof3} we obtain
\[\begin{split}
 \mathop {\lim \inf }\limits_{|y| \to \infty }& \big( |y|^{-p} ({u_{x_0,\overline \lambda (x_0)}} - u)(y) \big) \\
\geqslant & \mathop {\lim \inf }\limits_{|y| \to \infty } \int_{|z - x_0| \geqslant \overline \lambda (x_0)} |y|^{-p} k( x_0,\overline \lambda (x_0);y,z ) \Big[v{(z)^{ - q}} - {v_{x_0,\overline \lambda (x_0)}}{(z)^{ - q}}\Big]dz \\
 \geqslant& \int_{|z - x_0| \geqslant \overline \lambda (x_0)} 
\left\{ \begin{gathered}
\Big( {{{\Big( {\frac{{|z|}}{{\overline \lambda (x_0)}}} \Big)}^p} - 1} \Big) \times  \hfill \\
\bigg[v{(z)^{ - q}} - {\left( {\frac{{\overline \lambda (x_0)}}{{|z - x_0|}}} \right)^{2n - pq + p}}{v_{x_0,\overline \lambda (x_0)}}{(z)^{ - q}} \bigg] \hfill \\ 
\end{gathered}  \right\}dz > 0.
\end{split} \]
As a consequence, outside some large ball, we would have $({u_{x_0,\overline \lambda (x_0)}} - u)(y) \gtrsim |y|^p$ while in that ball and outside of $B(x_0, \overline \lambda (x_0) + 1)$ we would also have $(u_{x_0,\overline \lambda (x_0)} - u)(y) \gtrsim |y|^p$, given the smoothness of $u_{x_0,\overline \lambda (x_0)} - u$ and our assumption $u_{x_0,\overline \lambda (x_0)}(y) > u(y)$. Therefore, there exists some $\varepsilon_1 >0$ such that
\[({u_{x_0,\overline \lambda (x_0)}} - u)(y) \geqslant {\varepsilon _1}|y|^p\]
for all $|y-x_0| \geqslant \overline \lambda (x_0) + 1$. Recall that ${u_{x_0,\lambda }}(y) = (|x_0 - y|/\lambda)^p u({y^{x_0,\lambda }})$; hence there exists some $\varepsilon_2 \in (0, \varepsilon_1)$ such that
\begin{equation}\label{eqProof4}
\begin{split}
 (u_{x_0,\lambda } - u)(y) =& ({u_{x_0,\overline \lambda (x_0)}} - u)(y) + ({u_{x_0,\lambda }} - {u_{x_0,\overline \lambda (x_0)}})(y) \\
 \geqslant& {\varepsilon _1}|y|^p + ({u_{x_0,\lambda }} - {u_{x_0,\overline \lambda (x_0)}})(y) \geqslant \frac{{{\varepsilon _1}}}{2}|y|^p 
\end{split}
\end{equation}
for all $|y-x_0| \geqslant \overline \lambda (x_0) + 1$ and all $\lambda \in (\overline \lambda (x_0),\overline \lambda (x_0) + {\varepsilon _2})$. Repeating the above arguments shows that \eqref{eqProof4} is also valid for $v_{x_0,\lambda } - v$, that is
\begin{equation}\label{eqProof4-ForV}
\begin{split}
 (v_{x_0,\lambda } - v)(y) \geqslant \frac{{{\varepsilon _1}}}{2}|y|^p 
\end{split}
\end{equation}
for all $|y-x_0| \geqslant \overline \lambda (x_0) + 1$ and all $\lambda \in (\overline \lambda (x_0),\overline \lambda (x_0) + {\varepsilon _2})$ for a possibly new constants $\varepsilon_1 $ and $\varepsilon_2$.

\medskip\noindent\textbf{Estimate of $u_{x_0,\lambda } - u$ inside $B(x_0,\overline \lambda (x_0) + 1)$.} For $\varepsilon \in (0,{\varepsilon _2})$ which will be determined later, $\lambda \in (\overline \lambda (x_0),\overline \lambda (x_0) + \varepsilon ) \subset (\overline \lambda (x_0),\overline \lambda (x_0) + {\varepsilon _2})$, and $\lambda \leqslant |y-x_0| \leqslant \overline \lambda (x_0) + 1$, from \eqref{eqProof3}, we estimate
\[\begin{split}
 ({u_{x_0,\lambda }} - u)(y) \geqslant& \int_{|z - x_0| \geqslant \overline \lambda (x_0)} k (x_0,\lambda ;y,z)[v{(z)^{ - q}} - {v_{x_0,\lambda }}{(z)^{ - q}}]dz \\
 \geqslant &\int_{\overline \lambda (x_0) + 1 \geqslant |z - x_0| \geqslant \lambda } k (x_0,\lambda ;y,z)[v{(z)^{ - q}} - {v_{x_0,\lambda }}{(z)^{ - q}}]dz \\
 & + \int_{\overline \lambda (x_0) + 3 \geqslant |z - x_0| \geqslant \overline \lambda (x_0) + 2} k (x_0,\lambda ;y,z)[v{(z)^{ - q}} - {v_{x_0,\lambda }}{(z)^{ - q}}]dz \\
 \geqslant &\int_{\overline \lambda (x_0) + 1 \geqslant |z - x_0| \geqslant \lambda } k (x_0,\lambda ;y,z)[{v_{x_0,\overline \lambda (x_0)}}{(z)^{ - q}} - {v_{x_0,\lambda }}{(z)^{ - q}}]dz \\
 &+ \int_{\overline \lambda (x_0) + 3 \geqslant |z - x_0| \geqslant \overline \lambda (x_0) + 2} k (x_0,\lambda ;y,z)[v{(z)^{ - q}} - {v_{x_0,\lambda }}{(z)^{ - q}}]dz \\
=& I + II. 
\end{split} \]
As we shall see later, $I+II \geqslant 0$ provided $\varepsilon >0$ is sufficiently small. We now estimate $I$ and $II$ term by term.

\medskip\noindent\underline{Estimate of $II$}. From \eqref{eqProof4-ForV}, there exists $\delta_1>0$ such that $\big( v^{-q} - v_{x_0,\lambda}^{-q} \big)(z) \geqslant \delta_1$ for any $\overline \lambda (x_0) +2 \leqslant |z-x_0| \leqslant \overline \lambda (x_0) +3$. Note that by the definition of $k$ given in Lemma \ref{lem1} 
\[
k(x_0,\lambda; y,z)=k(0,\lambda; y-x_0,z -x_0)
\]
and from \eqref{eqGradientOfKAtX} there holds
\[(\nabla _y k) (0,\lambda ;y,z) \cdot y \big|_{|y| = \lambda } = p |y -z| ^{p-2} \big(|z|^2 - |y|^2\big) > 0\]
for all $\overline \lambda (x_0) +2 \leqslant |z| \leqslant \overline \lambda (x_0) +3$. Hence, there exists some constant $\delta_2>0$ independent of $\varepsilon$ such that
\[k(0,\lambda ;y,z) \geqslant {\delta _2}(|y| - \lambda )\]
for all $\overline \lambda (x_0) \leqslant \lambda \leqslant |y| \leqslant \overline \lambda (x_0) + 1$ and all $\overline \lambda (x_0) + 2 \leqslant |z| \leqslant \overline \lambda (x_0) + 3$. By replacing $y$ with $y-x_0$ and $z$ with $z-z_0$, and making use of the rule $k(x_0,\lambda; y,z)=k(0,\lambda; y-x_0,z -x_0)$, we obtain the same constant $\delta_2>0$ for the following estimate
\[k(x_0,\lambda ;y,z) \geqslant {\delta _2}(|y-x_0| - \lambda )\]
for all $\overline \lambda (x_0) \leqslant \lambda \leqslant |y-x_0| \leqslant \overline \lambda (x_0) + 1$ and all $\overline \lambda (x_0) + 2 \leqslant |z-x_0| \leqslant \overline \lambda (x_0) + 3$. Thus, we have
\begin{equation}\label{eqEstimateII}
II \geqslant \delta_1 \delta_2 (|y-x_0|-\lambda) \int_{\overline \lambda (x_0) + 3 \geqslant |z - x_0| \geqslant \overline \lambda (x_0) + 2} dz.
\end{equation}

\medskip\noindent\underline{Estimate of $I$}. To estimate $I$, we first observe that
\[|{v_{x_0,\overline \lambda (x_0)}}^{ - q} - {v^{ - q}}|(z) \lesssim \lambda - \overline \lambda (x_0) \lesssim \varepsilon \]
for all $\overline \lambda (x_0) \leqslant \lambda \leqslant |z-x_0| \leqslant \overline \lambda (x_0) + 1$ and all $\overline \lambda (x_0) \leqslant \lambda \leqslant \overline \lambda (x_0) + \varepsilon$. Also,
\[\begin{split}
 \int_{\lambda \leqslant |z-x_0| \leqslant \overline \lambda(x_0) + 1} & {k(x_0,\lambda ;y,z)dz} \\
= & \int_{\lambda \leqslant |z| \leqslant \overline \lambda (x_0)+ 1} {k(0,\lambda ;y-x_0,z)dz}\\
\leqslant & \int_{\lambda \leqslant |z| \leqslant \overline \lambda (x_0)+ 1} \Big| \Big( \frac{|y-x_0|}{\lambda}\Big)^p-1\Big| { |(y-x_0)^{0,\lambda } - z|^p dz} \\
& + \int_{\lambda \leqslant |z| \leqslant \overline \lambda (x_0)+ 1} {(|(y-x_0)^{0,\lambda } - z|^p - |(y-x_0) - z|^p )dz} \\
 \leqslant& C(|y-x_0| - \lambda ) + C|{(y-x_0)^{0,\lambda }} - (y-x_0)| \\
 \leqslant & C(|y-x_0| - \lambda ).
\end{split} \]
where $C>0$ is a constant independent of $\varepsilon$.
Thus, we obtain
\begin{equation}\label{eqEstimateI}
I \geqslant -C\varepsilon \int_{\overline \lambda (x_0) + 1 \geqslant |z - x_0| \geqslant \lambda } k (x_0,\lambda ;y,z) dz.
\end{equation}
By combining \eqref{eqEstimateI} and \eqref{eqEstimateII}, it follows that for some sufficiently small $\varepsilon>0$ we have
\[\begin{split}
 ({u_{x_0,\lambda }} - u)(y) \geqslant& \bigg( \delta_1 \delta_2 \int_{\overline \lambda (x_0) + 3 \geqslant |z - x_0| \geqslant \overline \lambda (x_0) + 2} dz -C\varepsilon \bigg) (|y-x_0|-\lambda) \geqslant 0
\end{split} \]
for $\overline \lambda (x_0) \leqslant \lambda \leqslant \overline \lambda (x_0) + \varepsilon$ and $\lambda \leqslant |y-x_0| \leqslant \overline \lambda (x_0) + 1$. 

\medskip\noindent\textbf{Estimates of $u_{x_0,\lambda } - u$ and $v_{x_0,\lambda } - v$ when $|y-x_0| \geqslant \overline \lambda (x_0) + 1$.} Combining the preceding estimate for $u_{x_0,\lambda } - u$ in the ball $B(x_0, \overline \lambda (x_0) + 1)$ and \eqref{eqProof4} gives
\[\begin{split}
 ({u_{x_0,\lambda }} - u)(y) \geqslant 0
\end{split} \]
for $\overline \lambda (x_0) \leqslant \lambda \leqslant \overline \lambda (x_0) + \varepsilon$ and $\lambda \leqslant |y-x_0|$. By repeating the procedure above for the difference $v_{x_0,\lambda} -v$, we can conclude that 
$$(v_{x_0,\lambda} -v)(y) \geqslant 0$$ 
for $\overline \lambda (x_0) \leqslant \lambda \leqslant \overline \lambda (x_0) + \varepsilon$ and $\lambda \leqslant |y-x_0|$ where $\varepsilon$ could be smaller if necessary; thus giving us a contradiction to the definition of $\overline\lambda(x_0)$.
\end{proof}

In the last lemma, we will prove that $\overline \lambda (x) <\infty$ everywhere in $\Rset^n$ whenever $\overline \lambda (x_0) <\infty$ for some point $x_0 \in \Rset^n$.

\begin{lemma}\label{lemLamdaVanishAtEveryPoint}
If $\overline \lambda (x_0) <\infty$ for some point $x_0 \in \Rset^n$ then $\overline \lambda (x) <\infty$ for any point $x \in \Rset^n$; hence
\[
u_{x,\overline \lambda (x)} \equiv u \quad \text{ and } \quad {v_{x,\overline \lambda (x)}} \equiv v
\]
for all $x \in \Rset^n$ in $\Rset^n$.
\end{lemma}

\begin{proof}
Suppose that there exists some $x_0 \in \Rset^n$ such that $\overline\lambda(x_0)<\infty$, then by Lemma \ref{lem3} and for $|y|$ sufficiently large, we have
\begin{align*}
 {\left| y \right|^{-p} }u(y) 
& =  \left| y \right|^{-p}  {\Big( {\frac{{\overline \lambda ({x_0})}}{{\left| {y - {x_0}} \right|}}} \Big)^{-p} }u\Big( {{x_0} + \lambda {{({x_0})}^2}\frac{{y - x_0}}{{{{\left| {y - x_0} \right|}^2}}}} \Big) \\
 &= \overline \lambda {({x_0})^{-p}}{\Big( {\frac{{\left| y- x_0 \right|}}{{\left| {y } \right|}}} \Big)^p }u\Big( {{x_0} + \lambda {{({x_0})}^2}\frac{{y - x_0}}{{{{\left| {y - x_0} \right|}^2}}}} \Big) .
\end{align*}
This implies
\begin{equation}\label{eq14}
\mathop {\lim }\limits_{\left| y \right| \to \infty } {\left| y \right|^{-p} }u(y) = \overline \lambda {({x_0})^{-p}}u({x_0}).
\end{equation}
By repeating the same argument, we obtain
\begin{equation}\label{eq15}
\mathop {\lim }\limits_{\left| y \right| \to \infty } {\left| y \right|^{-p} } v(y) = \overline \lambda {({x_0})^{-p}} v(x_0).
\end{equation}
Let $x \in \Rset^n$ be arbitrary. By the definition of $\overline\lambda(x)$ we get that $u_{x,\lambda } (y) \geqslant u(y)$ and $v_{x,\lambda } (y) \geqslant v(y)$ for all $0 < \lambda < \overline \lambda (x)$ and all $x,y$ such that $ | y - x | \geqslant \lambda$. Then by a direct computation and using \eqref{eq14}, we can easily see that
\begin{equation}\label{eq16}
\begin{split}
\mathop {\lim \inf }\limits_{\left| y \right| \to \infty } {\left| y \right|^{-p}}u(y) &\leqslant \mathop {\lim \inf }\limits_{\left| y \right| \to \infty } {\left| y \right|^{-p}}{u_{x,\lambda }}(y) \\
&= \mathop {\lim \inf }\limits_{\left| y \right| \to \infty } {\left| y \right|^{-p} }{\Big( {\frac{{\lambda }}{{\left| {y - x} \right|}}} \Big)^{-p} }u\Big( {x + \lambda^2 \frac{{y - x }}{{{{\left| {y - x } \right|}^2}}}} \Big) \\
&= \lambda ^{-p} u(x)
\end{split}
\end{equation}
for all $0 < \lambda < \overline \lambda (x)$. Combining \eqref{eq14} and \eqref{eq16}, we obtain $ \overline \lambda (x_0)^{-p} u( x_0 ) \leqslant \lambda ^{-p} u(x)$ for all $0 < \lambda < \overline \lambda (x)$. Therefore, $\overline \lambda (x) < +\infty$ for all $x \in {\Rset^n}$ as claimed.
\end{proof}

\subsection{Proof of Theorem \ref{thmCLASSIFICATION}}

To conclude Theorem \ref{thmCLASSIFICATION}, we first recall the following two lemmas from \cite{l2004}. These two lemmas have been used repeatedly in many works related to the underlying problem.

\begin{lemma}\label{lemKey1}
For $\nu \in \Rset$ and $f$ a function defined on $\Rset^n$, valued in $[-\infty, +\infty]$ let
\[
{\Big( {\frac{\lambda }{{\left| {y - x} \right|}}} \Big)^\nu }f\Big( {x + {\lambda ^2}\frac{{y - x}}{{{{\left| {y - x} \right|}^2}}}} \Big) \leqslant f(y)
\]
for all $x,y$ satisfying $| x - y | > \lambda > 0$. Then $f$ is constant or is identical to infinity.
\end{lemma}

\begin{lemma}\label{lemKey2}
For $\nu \in \Rset$ and $f$ a continuous function in $\Rset^n$. Suppose that for every $x\in \Rset^n$, there exists $\lambda(x)>0$ such that
\[{\Big( {\frac{{\lambda (x)}}{{\left| {y - x} \right|}}} \Big)^\nu }f\Big( {x + \lambda {{(x)}^2}\frac{{y - x}}{{{{\left| {y - x} \right|}^2}}}} \Big) = f(y)\]
for all $y \in {\Rset^n} \setminus \{ x\}$. Then for some $a \geqslant 0$, $d>0$ and $\overline x \in \Rset^n$ 
\[f(x) = \pm a \big(d + | x - \overline x |^2 \big)^{-\nu /2} .\]
\end{lemma}

To prove Theorem \ref{thmCLASSIFICATION}, we will consider the following two possible cases:

\smallskip\noindent\textbf{Case 1.} If $\overline\lambda(x)=\infty$ for any $x \in \Rset^n$, then $u_{x,\lambda }(y) \geqslant u(y)$ for all $\lambda > 0$ and for any $x, y$ satisfying $|y-x| \geqslant \lambda$. By Lemma \ref{lemKey1}, $u$ must be a constant. Similarly, $v$ is also a constant. However, this is not the case since solutions of \eqref{eqIntegralSystem} cannot be constant. 

\smallskip\noindent\textbf{Case 2.} If there exists some $x_0 \in \Rset^n$ such that $\overline\lambda(x_0)<\infty$, then by Lemma \ref{lemLamdaVanishAtEveryPoint}, we deduce that $\overline\lambda(x)<\infty$ for any point $x \in \Rset^n$. By Lemma \ref{lemKey2}, we express $u$ as
\begin{equation}\label{eqForm4U}
u(x) = a_1 ( b_1^2 + | {x - \overline x_1 } | ^2)^{p/2}
\end{equation}
for some $a_1, d_1>0$ and some point $\overline x_1 \in \Rset^n$. Similarly, $v$ can be expressed as
\begin{equation}\label{eqForm4V}
v(x) = a_2 ( b_2^2 + | {x - \overline x_2 } | ^2)^{p/2}
\end{equation}
for some $a_2, d_2>0$ and some point $\overline x_2 \in \Rset^n$. To realize that $u \equiv v$, we observe that $u$ given in \eqref{eqForm4U} satisfies the following equation
\[
 u(x) = \int_{\Rset^n} { |x-y|^p u(y)^{-(1+2n)/p} dy};
\]
see \cite[Appendix A]{l2004}. Using the above equation for $u$ and \eqref{eqIntegralSystem}, we must have $u \equiv v$ in $\Rset^n$ and hence we conclude that
\[ 
u(x)=v(x) = a(b^2 + |x-\overline x|^2)^{p/2}
\]
for some constants $a, b>0$ and some $\overline x \in \Rset^n$ as claimed.

\subsection{Proof of Theorem \ref{thmOPTIMAL}}

Theorem \ref{thmOPTIMAL} follows immediately from Theorem \ref{thmCLASSIFICATION}. The reason is because the sharp constant $\Cscr_{n,\lambda}$ as stated in theorem can also be computed using the precise form of the optimal functions established in Theorem \ref{thmCLASSIFICATION}. For this reason, we will omit the proof and refer interested readers to \cite[Section 3.2.2]{dz2014}. 

\subsection{The limiting case of the reversed HLS inequality (\ref{eq:RHLS-diagonal})}

Let us now consider the limiting case $\lambda = 0$ in \eqref{eq:RHLS-diagonal}. Clearly for this case, $2n/(2n+\lambda)=1$ and hence $\Cscr_{n, 0}=1$ is also sharp since
\begin{equation}\label{eq:RHLS-diagonal-1-1}
\int_{\Rset^n} \int_{\Rset^n} f(x) g(y) dx dy = \|f\|_{L^1(\Rset^n)} \|g\|_{L^1(\Rset^n)}.
\end{equation}
For each $\lambda >0$, we combine \eqref{eq:RHLS-diagonal} and \eqref{eq:RHLS-diagonal-1-1} to get
\begin{equation}\label{eq:RHLS-diagonal-1-2}
\begin{split}
\int_{\Rset^n} \int_{\Rset^n}  f(x) &\frac{|x-y|^\lambda - 1}{\lambda } g(y) dx dy \\
\geqslant & \frac{1}{\lambda} 
\left\{ \begin{gathered}
   \Cscr_{n, \lambda}  \|f\|_{L^\frac{2n}{2n+\lambda}(\Rset^n)} \|g\|_{L^\frac{2n}{2n+\lambda}(\Rset^n)} - \\
  \|f\|_{L^1(\Rset^n)} \|g\|_{L^1(\Rset^n)} \\ 
\end{gathered}  \right\},
\end{split}
\end{equation}
where the constant $\Cscr_{n, \lambda}$ given in Theorem \ref{thmOPTIMAL} is as follows
\[
\Cscr_{n,\lambda} ={\pi ^{\lambda /2}}\frac{{\Gamma (n/2 - \lambda /2)}}{{\Gamma (n - \lambda /2)}}{\left( {\frac{{\Gamma (n)}}{{\Gamma (n/2)}}} \right)^{1-\lambda /n}} .
\]
Taking the limit under the integral sign in \eqref{eq:RHLS-diagonal-1-2} as $\lambda \searrow 0$, we first obtain
\begin{equation}\label{eq:RHLS-diagonal-1-3}
\begin{split}
-\int_{\Rset^n}  \int_{\Rset^n} f(x) &\log|x-y| g(y) dx dy \\
\geqslant & \lim_{\lambda \searrow 0}  \frac{1}{\lambda} 
\left\{ \begin{gathered}
  \Cscr_{n, \lambda}\|f\|_{L^\frac{2n}{2n+\lambda}(\Rset^n)} \|g\|_{L^\frac{2n}{2n+\lambda}(\Rset^n)} - \\
  \|f\|_{L^1(\Rset^n)} \|g\|_{L^1(\Rset^n)}  \\ 
\end{gathered}  \right\} \\
= & \lim_{\lambda \searrow 0} \frac{\Cscr_{n, \lambda} -1 }{\lambda }  \|f\|_{L^\frac{2n}{2n+\lambda}(\Rset^n)} \|g\|_{L^\frac{2n}{2n+\lambda}(\Rset^n)} \\
& + \lim_{\lambda \searrow 0} \frac 1\lambda 
\left\{ \begin{gathered}
   \|f\|_{L^\frac{2n}{2n+\lambda}(\Rset^n)} \|g\|_{L^\frac{2n}{2n+\lambda}(\Rset^n)} -   \\
  \|f\|_{L^1(\Rset^n)} \|g\|_{L^1(\Rset^n)} \\ 
\end{gathered}  \right\}.
\end{split}
\end{equation}
By denoting $\Cscr_{n,0}^\star = \lim_{\lambda \searrow 0} (\Cscr_{n, \lambda} -1 )/\lambda$, which can be easily computed explicitly, the first term on the right most of \eqref{eq:RHLS-diagonal-1-3} becomes $\Cscr_{n,0}^\star \|f\|_{L^1(\Rset^n)} \|g\|_{L^1(\Rset^n)}$. For the remaining terms, the calculation is a bit more tedious; however, after long computations, we get
\[\begin{split}
 \frac{1}{2n}\big( \|f\|_{ L^1(\Rset^n) } &\log \|f\|_{ L^1(\Rset^n) } - \|f\log f \|_{ L^1(\Rset^n) } \big) \|g\|_{ L^1(\Rset^n) } \hfill \\
 &+ \frac{1}{2n}\big( \|g\|_{ L^1(\Rset^n) }\log \|g\|_{ L^1(\Rset^n) } - \|g\log g\|_{ L^1(\Rset^n) } \big) \|f\|_{ L^1(\Rset^n) }. 
\end{split} \]
Formally, we obtain the following reversed log-HLS inequality 
\begin{equation}\label{eq:RHLS-diagonal-1-4}
\begin{split}
-\int_{\Rset^n} \int_{\Rset^n} & f(x) \log|x-y| g(y) dx dy \geqslant \Cscr_{n,0}^\star \|f\|_{L^1(\Rset^n)} \|g\|_{L^1(\Rset^n)} \\
 &+ \frac{1}{2n} \big( \|f\|_{ L^1(\Rset^n) } \log \|f\|_{ L^1(\Rset^n) } - \|f\log f \|_{ L^1(\Rset^n) } \big) \|g\|_{ L^1(\Rset^n) } \hfill \\
 &+ \frac{1}{2n}\big( \|g\|_{ L^1(\Rset^n) }\log \|g\|_{ L^1(\Rset^n) } - \|g\log g\|_{ L^1(\Rset^n) } \big) \|f\|_{ L^1(\Rset^n) }.
\end{split}
\end{equation}
The above formal derivation requires some conditions for $f$ and $g$ in order for \eqref{eq:RHLS-diagonal-1-4} to hold. In view of \cite[Theorem 1]{CarlenLoss92}, one possible assumption of $f$ and $g$ could be $f, g \in L^1 (\Rset^n)$ with $ f(x) \log (1+|x|^2) \in L^1 (\Rset^n)$ and $g(x) \log (1+|x|^2) \in L^1 (\Rset^n)$. We do not treat this issue in the present paper and leave it for interested readers.

\section*{Acknowledgments}

Q.A.N. would like to thank Ninh Van Thu and Do Duc Thuan for their interest and useful discussion, especially on Section 4 of the paper. He would also like to express his gratitude to the Vietnam Institute for Advanced Study in Mathematics (VIASM) for hosting where part of this paper was prepared and announced in \cite{Ngo2015}. V.H.N. would like to thank the European Research Council for providing him the support from the research grant numbered 305629. Lastly, both authors would like to thank Eunice Chew Shuhui for her careful proofreading of this paper.

\end{document}